\documentclass{amsart}
\usepackage{placeins}
\usepackage{amsmath,amssymb}
\usepackage{textcomp}
\usepackage{stmaryrd}
\usepackage{amsmath}
\usepackage{enumitem}
\usepackage{makecell}
\usepackage{enumerate}
\usepackage{blindtext}
\extrafloats{100}
\usepackage{booktabs}
\usepackage{bm}
\usepackage[caption=false]{subfig}

\usepackage{algorithm,algcompatible,amsmath}
\algnewcommand\OUTER{\item[\textbf{OUTER CYCLE}]}%
\algnewcommand\INNER{\item[\textbf{INNER CYCLE:}]}%
\usepackage{graphicx}
\usepackage{tabularx,ragged2e,booktabs,caption}
\newcolumntype{C}[1]{>{\Centering}m{#1}}

\usepackage{geometry}
\usepackage{subfloat}
\DeclareMathOperator*{\argmin}{argmin}

\numberwithin{equation}{section}
\theoremstyle{plain}
\usepackage{cleveref}
  \newtheorem{theorem}{Theorem}[section]

\theoremstyle{remark}

\theoremstyle{defn}
\newtheorem{proposition}[theorem]{Proposition}
\newtheorem{lem}[theorem]{Lemma}

\newtheorem{remark}{Remark}[section]

\newtheorem{oss}{Remark}

\def\to{\rightarrow}
\def\eps{\varepsilon}

\newcommand{\R}{{\mathbb R}}

\title[]{On a monotone scheme for nonconvex nonsmooth optimization with applications to fracture mechanics}
\thanks{This work was  supported by the ERC advanced grant $668998$ (OCLOC) under the EU's H$2020$ research programme. }

\author{Daria Ghilli}
 \address{Daria Ghilli, University of Graz, Institute of Mathematics and Scientific Computing,  Universit{\"a}tsplatz 3, Austria}
\email{daria.ghilli@uni-graz.at}

\author{Karl Kunisch}
\address{Karl Kunisch, University of Graz, Institute of Mathematics and Scientific Computing,  Universit{\"a}tsplatz 3, Austria \\
Johann Radon Institute for Computational and Applied Mathematics (RICAM), Austrian Academy of Sciences, Altenbergerstrasse 69, Linz, Austria\\}
\email{karl.kunisch@uni-graz.at}

\begin{document}

\maketitle

\begin{abstract}
A general class of   nonconvex optimization problems  is considered, where the penalty is the composition of a linear operator with a nonsmooth nonconvex mapping, which is concave on the positive real line.   The necessary optimality condition of a regularized version of the original problem is solved by means of a monotonically convergent
scheme.
Such problems arise in  continuum mechanics, as for instance cohesive fractures, where singular behaviour is usually modelled by nonsmooth nonconvex energies. The  proposed algorithm is successfully tested for  fracture mechanics problems. Its performance is also compared to two alternative algorithms for nonsmooth nonconvex optimization arising in optimal control and mathematical imaging.
\end{abstract}



\section{Introduction}

In this paper we investigate a class of nonconvex and nonsmooth optimization problems, where the penalty is the composition of a nonsmooth nonconvex mapping with a linear operator and the smooth part is a least squares type   term.

Similar optimization problems  in the case where the operator inside the penalty coincides with the identity matrix have attracted increasingly attention due to their applications to sparsity of solutions, feature selection, and many other related fields as e.g. compressed sensing, signal processing, and machine learning (see e.g. \cite{7,14}).  The convex nonsmooth case of the $\ell^1$ norm has gained large popularity and has been thoroughly studied. The convexity allows to formulate efficient and globally convergent algorithms to find a numerical solution.  Here we mention \cite{10,46} where the basis pursuit and the Lasso problems were  introduced to solve $\ell^1$ minimization problems.

Recently  increased interest has arisen towards nonconvex and nonsmooth penalties, such as  the $\ell^\tau$ quasi-norm, with $\tau$ larger or equal to zero and less than $1$ (see e.g. \cite{BRLORE,HIWU13,KI,KKR,LI,OHDOBRPO15}), the smoothly clipped absolute deviation (SCAD) \cite{FL,JJLR}, and the minimax concave penalty (MCP) \cite{CHZ,JJLR}. The nonconvexity  has been shown to provide some advantages with respect to the convex models. For example, it allows to require less data in order to recover exactly the solution (see e.g. \cite{CS,FLai,45}) and it tends to produce unbiased estimates for large coefficients \cite{51,FL,FP}.
Note that all the previously mentioned works deal with the particular case where the operator coincides with the identity. 

Nonconvex optimization  problems as we consider, where the operator inside the penalty is different form the identity, arise also in the modelling of cohesive fractures in continuum mechanics, where the concavity of the penalty is crucial to model the evolution of the fracture energy released within the growth of the crack opening. Here the operator  is  of  importance to model  the jump of the displacement between the two lips of the fractures. We refer to \cite{PI13,GK,A1,A2} and subsection \ref{subsecnum} for more details.

 The study of these problems for  nonconvex penalties, including  $\ell^\tau$, with $\tau$ strictly positive and less than $1$, the SCAD and the MCP functionals, and for  linear operators  not necessarily coinciding with the identity, is also motivated by applications different from those arising in fracture mechanics. For example in imaging the $\ell^\tau$ quasi-norm,  with $\tau$ strictly positive and less than $1$, of the numerical gradient of the solution has been  proposed  as a nonconvex extension of the total  variation (like TV) regularizer  (see e.g \cite{HIWU13,OHDOBRPO15}) in order to reconstruct piecewise smooth solutions.
The SCAD and the MCP penalties have been used  for high dimensional regression and variable selection  methods in high-throughput biomedical studies \cite{BRHU}. We mention also that  the SCAD has  been proposed as a nonconvex penalty in the network estimation to attenuate the bias problem \cite{FANFENG}.

The main difficulties in the analysis of these problems  come from the interplay between the nonsmoothness, the nonconvexity, and the coupling between coordinates which is described by the operator inside the penalty. Since standard algorithms are not readily available, the  resolution of these problems  requires the development of new analytical and numerical techniques.

In the present paper we propose a monotonically  convergent algorithm to solve this kind of problems. This is an iterative procedure which solves the necessary optimality condition of a regularized version of the original problem.
A remarkable property of our scheme is the strict monotonicity of the functional along the sequence of iterates. The convergence of the iteration procedure is proved under the same assumptions that guarantee existence of solutions.

The performance of the scheme is successfully tested to simulate the evolution of cohesive fractures for several different test configurations. Then we turn to an issue of high relevance, namely the comparison between two alternative algorithms,  the GIST "General Iterative Shrinkage and Thresholding" algorithm for $\ell^\tau$ minimization,  with $\tau$ strictly positive and less than $1$  and  the FISTA "Fast Iterative Shrinkage-Thresholding Algorithm"  for $\ell^1$ minimization. The comparison is carried out with respect to the infimal value reached by the iteration procedure and with respect to computing time. Our results show that the monotone algorithm is able to reach a smaller value of the objective functional that we consider when compared to the one of GIST. 
Note that, differently from  GIST, the monotone scheme solves a system of nonlinear equations at each iteration level.
We remark that in \cite{LLSZ}  GIST was compared with the IRLS "iterative reiweighted least squares" algorithm, which is another popular scheme for $\ell^\tau$ minimization,  with $\tau$ strictly positive and less than $1$. The results of \cite{LLSZ} show that  GIST and IRSL have nearly the same performance, with only one difference which is speed, where GIST appears to be the faster one.

An analogous procedure to the one proposed in the present paper was developed in \cite{GK}   to solve similar problems  where the nonconvex penalty coincides with $\ell^\tau$ quasi-norm,  with $\tau$ strictly positive and less than or equal to $1$. 
With respect to \cite{GK}, in the present paper  we deal with more general concave penalties. Moreover, we carry out several numerical experiments for diverse situations in  cohesive  fracture mechanics, comparing the behaviours for different concave penalties such as the SCAD, the MCP and the $\ell^\tau$ penalty,  with $\tau$ strictly positive and less than $1$.
Finally in the present paper we compare the performance of the scheme with that of  GIST.


Let us recall some further literature concerning nonconvex nonsmooth optimization of the type investigated in the present paper.
In \cite{JJLR,J} a primal-dual active set type algorithm  has been developed,  in the case the operator inside the penalty coincides with the identity.
For  more references in this case we refer to  \cite{GK}.
Concerning $\ell^\tau$ minimization,  with $\tau$ larger than or equal to zero and less than or equal to $1$ when  the operator is not the identity, other techniques have  recently been investigated. Here we mention iteratively reweighted convex majorization algorithms \cite{OHDOBRPO15}, alternating direction method of multiplier (ADMM) \cite{LI} and finally a Newton-type solution algorithm for a regularized version of the original problem \cite{HIWU13}.
Finally we recall the paper \cite{A1}, where a novel algorithm for nonsmooth nonconvex optimization with linear constraints is proposed, consisting of a generalization of the well-known nonstationary augmented Lagrangian method for convex optimization. The convergence to critical points is proved and several tests were made for  free-discontinuity variational models, such as the Mumford-Shah functional. The nonsmoothness considered in \cite{A1} does not allow singular behaviour of the type that the $\ell^\tau$ term,  with $\tau$ larger than or equal to zero and strictly less than $1$ does.

The paper is structured as follows. In Section \ref{existence}, subsection \ref{subsecass} we state the precise assumptions, in subsection \ref{subsecex} we prove existence for the problem in consideration, in subsection \ref{subsecmon} we propose the monotone scheme to solve a regularized version of the original problem and we prove its convergence, and finally in subsection \ref{subsecas} we study the asymptotic behavior as the concavity and regularization parameters go to zero. In Section \ref{alg} we present the precise form of our scheme. In subsection \ref{subsecnum} we discuss our numerical experience for cohesive evolution of fracture mechanics  and in subsection \ref{compg} we compare the performance of our scheme to that of GIST for three different test cases, the academical M-matrix example, an optimal control problem and a microscopy imaging example.

 \section{Existence and monotone algorithm}\label{existence}
\subsection{Assumptions}\label{subsecass}
We consider
\begin{equation}\label{optprobphi}
\min_{x \in \R^n}J(x)=\frac{1}{2}|A x-b|_2^2+\sum_{i=1}^r\phi(\Lambda x)_i,
\end{equation}
where  $A \in \mathbb{M}^{m\times n}$, $\Lambda \in \mathbb{M}^{r\times n},$ $b \in \R^m$   and $\phi(t): \R \to \R^+$ satisfies
\[
{\bf (H)}
\begin{cases}
\text{(i)}\; &\phi \mbox{ is even with } \phi(0)=0, \mbox{ nondecreasing for } t\geq 0 \mbox{ and continuous};\\
\text{(ii)}\; &\phi \mbox{ is differentiable on } ]0,\infty[;\\
\text{(iii)}\; &\phi \mbox{ is concave on } \R^+;\\
\text{(iv)}\; & \mbox{ there exists a neighbourhood of zero where the function } t \to \frac{\phi'(t)}{t} \mbox{ is monotone};\\
\end{cases}
\]
Above monotonically increasing or decreasing are admitted.
Throughout the rest of the paper we will use  the notation 
$$
\Phi(\Lambda x):=\sum_{i=1}^r\phi(\Lambda x)_i.
$$

Under assumption ${\bf (H)}$, the following two cases are analysed: 
\begin{itemize}
\item[(a)]
\begin{itemize}
\item[(i)]
$\phi(t)$ is a constant when $|t|\geq t_0$ for some $t_0>0$;
\item[(ii)]
A is coercive, i.e. $\mbox{rank}(A)=n$.
\end{itemize}
\vspace{0.2cm}

\item[(b)]
\begin{itemize}
\item[(i)]
for some $\gamma>0$ it holds $\phi(at)=a^\gamma \phi(t)$ for all $t \in \R$ and $a \in \R^+$;
\item[(ii)]
$
\mbox{Ker}(A)\cap\mbox{Ker}(\Lambda)=\{0\}.
$
\end{itemize}
\vspace{0.2cm}

\end{itemize}

Three popular examples of nonconvex penalties which satisfy ${\bf (H)}$ and the assumptions on $\phi$ in (a) or (b)  are the following:
\medskip
 \begin{description}
 \item[${\bm \ell^\tau}$]  $\tau \in (0,1], \lambda >0$
 \begin{equation}\label{p}
 \phi(t)=\lambda |t|^\tau,
 \end{equation}
 satisfying $(b)(i)$.
 \vspace{0.5cm}
 
 \item[SCAD]  $\tau >1,  \lambda >0$
 \begin{equation}\label{SCAD}
\phi(t)=\left\{
\begin{array}{lll}
\frac{\lambda^2(\tau+1)}{2} \quad &|t|\geq \lambda \tau\\
\frac{\lambda \tau |t|-\frac{1}{2}(t^2+\lambda^2)}{\tau-1}\quad &\lambda < |t|\leq \lambda \tau\\
\lambda |t| \quad &|t|\leq \lambda,
\end{array}
\right.\,
\end{equation}
 satisfying $(a)(i)$.
 \vspace{0.5cm}
 
 \item[MCP] $\tau>1,  \lambda >0$
 \begin{equation}\label{MCP}
\phi(t)=\left\{
\begin{array}{lll}
\lambda(|t|-\frac{t^2}{2\lambda\tau})\quad &|t|< \lambda \tau\\
\frac{\lambda^2\tau}{2}\quad &|t|\geq \lambda \tau,
\end{array}
\right.\,
\end{equation}
 satisfying $(a)(i)$.
 \end{description}
\begin{remark}\rm{
The singularity at the origin of the three penalties leads to sparsity of the solution. In the SCAD and the MCP, the derivative vanishes for large values to ensure unbiasedness. 

Problems as \eqref{optprobphi} with $\phi$ given by the $\ell^\tau$-quasi norm with $\tau \in (0,1)$ were studied in \cite{GK}. For more details on its statistical properties, such as variable selection and oracle property, of the $\ell^\tau$-quasi norm, we refer  to \cite{CS,FLai,HHM,KKF}. 

The SCAD (smoothly clipped absolute deviation) (\cite{FL,FP}) has raised interest in relation to variable selection consistency and asymptotic estimation efficiency (see \cite{FP}). It  can be obtained upon integration of the following formula for $\tau>2$
$$
\phi(t)=\lambda \int_0^{|t|}\min\left(1,\frac{\max(0,\lambda \tau-|s|)}{\lambda (\tau-1)}\right) ds.
$$
The MCP (minimax concave penalty) \cite{CHZ} can be recovered from the following formula
$$
\phi(t)=\lambda \int_0^{|t|} \max\left (0,1-\frac{|s|}{\lambda \tau}\right) ds
$$
and minimizes the maximum concavity $\sup_{0<t_1<t_2}\frac{\left(\phi'(t_1)-\phi'(t_2)\right)}{(t_2-t_1)}$ subject to the constraints $\phi'(t)=0$ for any $|t|\geq \lambda \tau$ (unbiasedness) and $\phi'(0^{\pm})=\pm \lambda$ (feature selection). The condition $\tau>1$ ensures the wellposedness of the thresholding operator.

}
\end{remark}

\subsection{Existence}\label{subsecex}
First we prove coercivity of the functional J in \eqref{optprobphi} under assumptions (a) or (b).
\begin{lem}\label{coer}
Let assumptions \textbf{(H)} and either (a) or (b) hold. Then the functional J in \eqref{optprobphi} is coercive.
\end{lem}
\begin{proof}
Under assumption (a), the coercivity of J follows trivially. Suppose now that $(c)$ holds. Then the result follows by similar arguments to that used in  \cite{GK}, Theorem  1 (where $\phi$ is the $\ell^\tau$ quasi-norm). We proceed by contradiction and we suppose that $|x_k|_2\to + \infty$ and $J(x_k)$ is bounded. For each $k$, let $x_k=t_kz_k$ be such that $t_k\geq 0, x_k \in \R^n$ and $|z_k|_2=1. $ By (b) (i) we have
$$
\Phi(\Lambda z_k)= \frac{1}{t_k^\gamma}\Phi(\Lambda x_k)
$$
and then since $t_k \to + \infty$ and $J(x_k)$ is bounded,  we have for $k \to + \infty$
$$
\displaystyle 0\leq |A z_k|_2^2+ \Phi(\Lambda z_k)=\frac{1}{ t_k^2}|A x_k|_2^2+ \frac{1}{t_k^\gamma}\Phi(\Lambda x_k)\leq \frac{1}{t_k^{\min\{2,\gamma\}}}\left( |A x_k|_2^2+\Phi(\Lambda x_k)\right) \to 0
$$
and hence
$$
\lim_{k \to + \infty} |Az_k|_2^2+\Phi(\Lambda z_k)=0.
$$
By compactness, the sequence $\{z_k\}$ has an accumulation point $\bar z$ such that $|\bar z|=1$ and $\bar z \in \mbox{Ker}(A)\cap \mbox{Ker}(\Lambda)$, which contradicts (c) (ii). 
\end{proof}
In the following theorem we state the existence of at least a minimizer to \eqref{optprobphi} under either (a) or (b). We omit the proof since it follows directly by the continuity and  coercivity of the functional in \eqref{optprobphi}.
\begin{theorem}\label{thmex}
Let assumptions \textbf{(H)} and either  (a) or  (b) hold. Then there exists at least one minimizer to problem \eqref{optprobphi}.
\end{theorem} 
\begin{oss}\rm{
We remark that when assumption (a) (i) holds but $A$ is not coercive, existence can still be proven in case  
$\Lambda \in \R^{n\times n}$ is invertible. Indeed by the invertibility of $\Lambda$, one can define $\bar{y}=\Lambda^{-1}\bar x$, where $\bar x$ is a minimizer  of $\bar{J}(x)=\frac{1}{2}|(A\Lambda^{-1})x-b|_2^2+\Phi(x)$ and prove that $\bar{y}$ is  a minimizer of \eqref{optprobphi}. The existence of a minimizer for the functional $\bar{J}$ was proven in \cite{J}, Theorem $2.1$.

However in our analysis we cover the two cases (a) and (b) since when (a) (ii) is replaced by the invertibility of $\Lambda$, we can not prove the coercivity of $J$, which is a key element for the convergence of the algorithm that we analise (see the following section). 
}
\end{oss}
\subsection{A monotone convergent algorithm}\label{subsecmon}
Following \cite{KI}, in order to overcome the singularity of the function $\phi(t)$ near $t=0$, we consider for $\eps>0$ the following regularized version of \eqref{optprobphi} 
\begin{equation}
\label{optprobepsphi}
\min_{x \in \R^n}J_\eps(x) = \frac{1}{2}|A x-b|_2^2+ \Psi_\eps(|\Lambda x|^2), 
\end{equation}
where for $t \geq 0$
\begin{equation}\label{psieps}
\Psi_{\eps}(t)= \left\{
\begin{array}{ll}
\frac{\phi'(\eps)}{2\eps}t+\left(1-\frac{\phi'(\eps) \eps}{2\phi(\eps)}\right)\phi(\eps) \quad &\mbox{for }\,\, 0\leq t \leq \eps^2\\
\noalign{\smallskip}
\phi(\sqrt{t}) \quad & \mbox{ for }\,\, t \geq \eps^2,
\end{array}
\right.\,
\end{equation}
and $\Psi_\eps(|\Lambda x|^2)$ is short for $\sum_{i=1}^r \Psi_\eps(|(\Lambda x)_i|^2)$. 
Note that 
\begin{equation}\label{derpsi}
\Psi'_\eps(t)=\frac{1}{\max\left\{\frac{2\eps}{\phi'(\eps)},\frac{2\sqrt{t}}{\phi'(\sqrt{t})}\right\}} \geq 0 \quad \mbox{ on } (0, \infty),
\end{equation}
hence $\Psi_\eps$ is $C^1$ and by assumption \textbf{(H)} (iii) is concave on $[0,\infty)$. Moreover $t \to \Psi_\eps(t^2) \in C^1(-\infty, \infty)$.
\begin{oss}
\rm{
In \eqref{derpsi} we suppose  that the function $x \to \frac{\phi'(x)}{x}$ is decreasing. When the function $x \to \frac{\phi'(x)}{x}$ is increasing, one needs to replace the maximum in \eqref{derpsi} with the minimum and the following results follow as well.
}
\end{oss}
 The necessary optimality condition for \eqref{optprobepsphi} is given by
\begin{equation}\label{necoptcondexp}
A^*A x+\Lambda^* \frac{1}{\max\left\{\frac{\eps}{\phi'(\eps)},\frac{|\Lambda x|}{\phi'(|\Lambda x|)}\right\}}\Lambda x=A^*b,
\end{equation}
 the second addend is short for the vector with $l$-component $\sum_{i=1}^r(\Lambda^*)_{li} \frac{1}{\max\left\{\frac{\eps}{\phi'(\eps)},\frac{|(\Lambda x)_i|}{\phi'(|(\Lambda x)_i|)}\right\}}(\Lambda x)_i$. For convenience of exposition in the following we write \eqref{necoptcondexp} in the more compact notation
$$
A^*A x+2\Lambda^* \Psi'_\eps(|\Lambda x|^2)\Lambda x=A^*b,
$$
where the $l$-component of the second addend is given by $\sum_{i=1}^r(\Lambda^*)_{li} \Psi'_\eps(|(\Lambda x)^2_i|)(\Lambda x)_i$.

This can equivalently be expressed as
\begin{equation}
\label{optcondeps2m}
A^*A x+2\Lambda^* \Psi'_\eps(|y|^2)y=A^*b \quad \mbox{with } y=\Lambda x.
\end{equation} 
In order to solve \eqref{optcondeps2m}, the following iterative procedure is considered:
\begin{equation}\label{iter2m}
A^*A x^{k+1}+2\Lambda^* \Psi'_\eps(|y^k|^2)y^{k+1}=A^*b \quad \mbox{where } y^{k}=\Lambda x^{k}.
\end{equation}
Existence of a solution to equation \eqref{iter2m} for any $k \geq 0$ follows from the existence of a minimizer of the associated optimization problem 
\begin{equation}\label{minit}
\min_{x\in \R^n} \frac{1}{2}|Ax-b|_2^2+\left|\sqrt{\Psi'_\eps(|y^k|^2)} \Lambda x\right|_2^2.
\end{equation}

We have the following convergence result.

\begin{theorem}\label{monotdec2m}
Assume \textbf{(H)} and either (a) or (b). For $\eps>0$, let $\{x_k\}$ be generated by \eqref{iter2m}. Then $J_{\eps}(x_k)$ is strictly monotonically decreasing, unless there exists some $k$ such that $x^k = x^{k+1}$, and $x^k$ satisfies the necessary optimality condition \eqref{optcondeps2m}. Moreover every  cluster point of $x^k$, of which there exists at least one, is a solution of \eqref{optcondeps2m}.
\end{theorem}

\begin{proof}

The proof strongly depends on the coercivity of the functional $J$ and it follows  arguments similar to those of \cite[Theorem $4.1$]{KI}. 

Multiplying \eqref{iter2m} by $x^{k+1}-x^k$, we get
\begin{eqnarray}\label{proofeq1}
\frac{1}{2}|A x^{k+1}|^2-\frac{1}{2}|A x^k|^2+\frac{1}{2}|A (x^{k+1}-x^{k})|^2&\nonumber +& \left( 2\Psi'_\eps(|y^k|^2)y^{k+1},y^{k+1}-y^{k}\right)\\  &=&(A^*b,x^{k+1}-x^k).
\end{eqnarray}
Note that 
\begin{equation}\label{mon1}
\left(y^{k+1},y^{k+1}-y^k\right)=\frac{1}{2}\sum_{i=1}^n\left(|y_i^{k+1}|^2-|y_i^{k}|^2+|y_i^{k+1}-y_i^{k}|^2\right).
\end{equation}
By assumption ${\bf (H)}\,(iii)$ the function $t \rightarrow \Psi_\eps(t)$ is concave on $[0,\infty)$, and we have 
\begin{equation}\label{mon3}
2\Psi_\eps(|y_i^{k+1}|^2)-2\Psi_\eps(|y_i^{k}|^2)-\Psi'_\eps(|y_i^k|^2)(|y_i^{k+1}|^2-|y_i^{k}|^2) \leq 0.
\end{equation}
Then, using \eqref{proofeq1}-\eqref{mon3}, we get
\begin{equation}\label{442m}
J_\eps(x^{k+1}) +\frac{1}{2}|A(x^{k+1}-x^k)|_2^2+\sum_i\Psi'_\eps(|y_i^k|^2)|y_i^{k+1}-y_i^{k}|^2 \leq J_\eps(x^k).
\end{equation}
From \eqref{442m} and the coercivity of $J_\eps$,  it follows that $\{x^k\}_{k=1}^\infty$ and thus $\{y^{k}\}_{k=1}^\infty$ are bounded. Consequently, from \eqref{442m} and \eqref{derpsi}, there exists a constant $\kappa>0$ such that 
\begin{equation}\label{45}
J_\eps(x^{k+1}) +\frac{1}{2}|A(x^{k+1}-x^k)|_2^2+\kappa|y^{k+1}-y^{k}|_2^2  \leq J_\eps(x^k).
\end{equation}
Conditions (a) (ii), (b) (ii) respectively imply that $J_\eps(x_k)$ is strictly decreasing unless $x^k=x^{k+1}$. In the latter case from \eqref{iter2m} we infer that $x^k$ solves \eqref{optcondeps2m}, from which we conclude the first  part of the theorem. 

From \eqref{45}, we conclude that 
\begin{equation}\label{46}
\sum_{k=0}^\infty |A(x^{k+1}-x^k)|_2^2+\kappa|y^{k+1}-y^{k}|_2^2  <\infty.
\end{equation}
Since $\{x^k\}_{k=1}^\infty$ is bounded, there exists a subsequence and $\bar x \in \R^n$ such that $x^{k_l}\rightarrow \bar x$. By \eqref{46} we get
$$
\lim_{k\to \infty}|A(x^{k+1}-x^k)|_2^2+\kappa|y^{k+1}-y^{k}|_2^2=0.
$$
Then by using the coercivity of A under assumption (a) and the fact that $\mbox{Ker}(A)\cap \mbox{Ker}(\Lambda)=\{0\}$ under assumption (b),  we conclude  that 
$$
\lim_{k\to \infty}(x^{k+1}-x^{k})=0
$$
and hence $x^{k_l+1} \rightarrow \bar x$.
We can now pass  to the limit with respect to $k$ in  \eqref{iter2m}, to obtain that $\bar x$ is a solution to \eqref{optcondeps2m}.

\end{proof}

In the following proposition we establish the convergence of \eqref{optprobepsphi} to \eqref{optprobphi} as $\eps$  goes to zero.
\begin{proposition}
Assume \textbf{(H)} and either (a) or (b). Denote by $\{x_\eps\}_{\eps>0}$ a solution to \eqref{optprobepsphi}. Then any cluster point of  $\{x_\eps\}_{\eps>0}$, of which there exists at least one,  is a solution of \eqref{optprobphi}.
\end{proposition}
\begin{proof}
From the coercivity of $J_\eps$ we have that $\{x_\eps\}_{\eps}$ is bounded for $\eps$ small. Hence there exists a subsequence  and $\bar x \in \R^n$ such that $x_{\eps_l} \rightarrow \bar x$. 

By property \textbf{(H)} (i) of $\phi$, we have
\begin{equation}\label{Hi}
\lim_{t\to 0 }\phi(t)=0 \quad \mbox{ and } \quad \phi'(t)\geq 0 \quad \forall t \geq 0.
\end{equation}
Moreover by the concavity of the function $\phi$ we have
$$
\phi(t)-\phi(s)\leq \phi'(s)(t-s) \quad \mbox{for } s \in(0,\infty), t \in [0,\infty)
$$
and by choosing $s=\eps$ and $t=0$ and by \eqref{Hi}, we get for $\eps$ small enough
\begin{equation}\label{psiepseps}
\phi'(\eps)\eps \to 0 \,\,\mbox{ as } \eps \to 0.
\end{equation}
By the definition of $\Psi_\eps$, \eqref{Hi} and \eqref{psiepseps} we obtain that $\Psi_\eps(t)$ converges uniformly to $\phi(\sqrt{t})$ as $\eps \to 0$, equivalently
$$
 \sup_{t\in [0,\infty)}\left|\Psi_\eps(t) -\phi(\sqrt{t})\right| \to 0\quad \mbox{ as } \eps \to 0,
$$
from which we obtain
\begin{equation}\label{limpsieps}
\Psi_\eps(|\Lambda x_\eps|^2)=\sum_{i=1}^r\Psi_\eps(|(\Lambda x_\eps)_i|^2)  \to \sum_{i=1}^r \phi(\Lambda x_\eps)_i=\Phi(\Lambda\bar x)\quad \mbox{ as } \eps \to 0.
\end{equation}
 Since $\{x_\eps\}_{\eps}$ solves \eqref{optprobepsphi}, by letting  $\eps\to 0$ and using \eqref{limpsieps}, we easily get that  $\bar x$ is a solution of \eqref{optprobphi}.
\end{proof}

\subsection{Asymptotic behaviour  as $\lambda \to 0^{+}$ and $\tau\to 0^{+}$ for the power law}\label{subsecas}
We  discuss the asymptotics as $\lambda$ and $\tau$ go to zero  in \eqref{optprobphi} for $\phi(t)=|t|^\tau, \tau \in (0,1]$, which we repeat for convenience
\begin{equation}\label{optprobphias}
\min_{x \in \R^n}\frac{1}{2}|A x-b|_2^2+\lambda |\Lambda x|_\tau^\tau,
\end{equation}
where $A, b, \Lambda$ are as in \eqref{optprobphi}, $\tau \in (0,1], \lambda >0$ and 
$$
|\Lambda x|_\tau^\tau=\sum_{i=1}^r |(\Lambda x)_i|^\tau.
$$

First we analyse the convergence as $\lambda \to 0$ for any fixed $\tau >0$.
We denote by $P$  the orthogonal projection of $\R^n$ onto $\mbox{Ker}(A^{*})$ and set $\tilde{b} = (I-P)b \in \mbox{Rg}(A)$. Then
$$
|Ax -b|_2^2=|Ax-\tilde{b}|_2^2+|Pb|_2^2.
$$
For $\tau>0$ fixed consider the problem
\begin{equation}\label{minnorm}
\min |\Lambda x|_\tau^\tau \quad  \mbox{ subject to } Ax -\tilde{b}=0,
\end{equation}

\begin{theorem}
Let $\mbox{rank}(A)=n$. For $\tau>0$ fixed, let $x_{\lambda}$ be a minimizer of \eqref{optprobphias} for any given $\lambda >0$. Let $\tilde{x} \in \R^n$ be such that $A\tilde{x}=\tilde{b}$. Then every  cluster point  of solutions $x_{\lambda}$ to \eqref{optprobphias} as $\lambda \to 0^{+}$ is a solution to \eqref{minnorm}.
\end{theorem}

\begin{proof}
By optimality of $x_\lambda$ we have
\begin{equation}\label{G}
\frac{1}{2}|Ax_{\lambda} - \tilde{b}|_2^2+\lambda |\Lambda x_\lambda|_\tau^\tau \leq \frac{1}{2}|A\tilde{x}-\tilde{b}|_2^2+\lambda |\Lambda \tilde{x}|_\tau^\tau=\lambda |\Lambda \tilde{x}|_\tau^\tau,
\end{equation}
from which we conclude that $\lim |Ax_{\lambda} -\tilde{b}|_2^2=0$ as $\lambda \to 0^{+}$.

Since $\mbox{rank}(A)=n$, the sequence $\{x_\lambda\}_{\lambda >0}$ is bounded in $\lambda$. Then there exists $\bar x \in \R^n$ and a subsequence $x_{\lambda_l} \to \bar x$ satisfying $A \bar x =\tilde{b}$. From \eqref{G} we have
$$
 |\Lambda x_\lambda|_\tau^\tau\leq |\Lambda \tilde{x}|_\tau^\tau \quad \mbox{ for all } \tilde x \mbox{ satisfying } A\tilde{x}=\tilde{b}. 
$$
Taking the limit as $\lambda \to 0^{+}$, we conclude that $\bar x$ is a solution to \eqref{minnorm}.

\end{proof}

Now we prove the convergence as $\tau \to 0$ for any fixed $\lambda >0$ of  \eqref{optprobphias}   to the related   $\ell^0$-problem
\begin{equation}\label{optprobphi0}
\min_{x \in \R^n}\frac{1}{2}|Ax-b|_2^2+\lambda |\Lambda x|_0,
\end{equation}
where for any $x \in \R^n$
$$
|x|_0=\sum_{k=1}^n|x_k|^0=\mbox{ number of nonzero elements of } x.
$$
The precise statement is given in the following theorem.  
\begin{theorem}
Let $\mbox{rank}(A)=n$ and let $\lambda >0$ be fixed. Then any cluster point (of which there exists at least one) of solutions $\{x_{\tau}\}$ to \eqref{optprobphias} converges as $\tau \to 0^{+}$ to a solution of \eqref{optprobphi0}.
\end{theorem}
\begin{proof}
$\mbox{Rank}(A)=n$ ensures the existence of a converging subsequence (denoted with the same symbol) $\{x_\tau\}\to \bar x$ for some $\bar x\in \R^n$. For any fixed $i \in \{1,\cdots, r\}$, denote $y_\tau=|(\Lambda x_\tau)_i|$ and $\bar y=|(\Lambda \bar x)_i|$ and notice that  $y_\tau \to \bar{y}$ as $\tau \to 0$. Then if $\bar y=0$, we can assume $y_\tau <1$ for $\tau$ enough small and we conclude
$$
y_\tau^\tau<y_\tau\to 0 \quad \mbox{ as } \tau \to 0.
$$
If $\bar y>0$ we have
$$
\log(y_\tau^\tau)=\tau\log(y_\tau) \to 0 \quad \mbox{ as } \tau \to 0 
$$
and thus
$$
y_\tau^\tau \to 1 \quad \mbox{ as } \tau \to 0.
$$
 By using the above arguments  for all $i=1,\cdots, r$, we have
$$
|(\Lambda x_\tau)_i|^\tau \to |(\Lambda \bar x)_i|^0 \quad \mbox{ as } \tau \to 0
$$
and then we conclude
\begin{equation}\label{limtau}
|\Lambda x_\tau|^{\tau}_\tau\to |\Lambda \bar x|_0 \quad \mbox{ as } \tau \to 0.
\end{equation}
By the optimality of $x_\tau$ we have
$$
\frac{1}{2}|Ax_\tau-b|_2^2+\lambda |\Lambda x_\tau|^{\tau}_\tau\leq \frac{1}{2}|Ax-b|_2^2+\lambda |\Lambda x|_\tau^\tau, \quad \mbox{ for all } x \in \R^n.
$$
 Then the proof follows by taking the limit $\tau\to 0$ and using \eqref{limtau} to obtain
$$
\frac{1}{2}|A\bar x-b|_2^2+\lambda |\Lambda \bar x|_0\leq \frac{1}{2}|Ax-b|_2^2+\lambda |\Lambda x|_0, \quad \mbox{ for all } x \in \R^n.
$$
\end{proof}

\section{Algorithm and numerical results}\label{alg}
For convenience we recall the algorithm in the following form.  

\begin{algorithm}[h!]
	\caption{Monotone algorithm with $\eps$-continuation strategy}
	\begin{algorithmic}[1]
		\STATE Initialize $x^0$, $\eps^0$,  and set $y^{0}=\Lambda x^0$. Set $k=0$;
		\REPEAT
		\STATE  Solve for $x^{k+1}$ 
		$$
		A^*A x^{k+1}+\Lambda^*\frac{1}{\max\left\{\frac{\eps}{\phi'(\eps)},\frac{|y^k|}{\phi'(|y^k|)}\right\}}\Lambda x^{k+1}=A^*b.
			$$
		\STATE Set $y^{k+1}=\Lambda x^{k+1}$.
		\STATE Set $k=k+1$.
		\UNTIL{the stopping criterion is fulfilled}.
		\STATE  Reduce $\eps$ and repeat 2.
	\end{algorithmic}
\end{algorithm}

%

\begin{remark}\rm{
Note that an $\eps$-continuation strategy is performed, that is, the procedure is performed for an initial value $\eps^0$ and then $\eps$ is decreased up to a certain value. More specifically, in all our experiments, $\eps$  is  initialized with $10^{-1}$ and decreased up to $10^{-12}$. \\
}
\end{remark}

\begin{remark}\rm{
The stopping criterion is based on the $l^\infty$-norm of the equation \eqref{optcondeps2m} and the tolerance is set to $10^{-3}$ in all the following examples, expect for the fracture problem where it is of the order of  $10^{-15}$.
}
\end{remark}

In the following  subsection we present our numerical results in cohesive fracture mechanics.   Then in subsection  \ref{compg} the performance of our algorithm is compared to two other schemes for nonconvex and nonsmooth optimization problems.

\subsection{Application to quasi-static evolution of cohesive fracture models}\label{subsecnum}

In this section we focus on the numerical realization of  quasi-static evolutions of cohesive fractures. This kind of problems require the minimization of an energy functional, which has two components: the elastic energy and the cohesive fracture energy. The underlying idea is that the fracture energy is released gradually with the growth of the crack opening.  The cohesive energy, denoted by $\theta$, is assumed to be a monotonic non-decreasing  function of the jump amplitude of the displacement,  denoted by $\llbracket u \rrbracket$. Cohesive energies were introduced independently by Dugdale  \cite{DG} and Barenblatt  \cite{BA}, we refer to \cite{PI13} for more details on the models.  Let us just remark that the two models differ mainly in the evolution of  the derivative $\theta'(\llbracket u\rrbracket)$, that is, the \textit{bridging force}, across a crack amplitude $\llbracket u \rrbracket$. In Dugdale's model this force keeps a constant value up to a critical value of the crack opening and then drops to zero. In Barenblatt's model, the dependence of the force on $\llbracket u \rrbracket$ is continuous and decreasing. \\
In this section we test the $\ell^\tau$-term $0<\tau<1$ as a  model for the cohesive energy. In particular,  the cohesive energy is not differentiable in zero and the bridging force goes to infinity when the jump amplitude goes to zero. Note also that the bridging force goes to zero when the jump amplitude goes to infinity.\\
we denote by $u\,:\, \Omega \rightarrow \R$ the displacement function. The deformation of the domain is given by an external force which we express in terms of an external displacement function $g\,:\,\Omega\times [0,T] \rightarrow \R$. We require that the displacement $u$ coincides with the external deformation, that is
$$
u|_{\partial \Omega}=g|_{\partial \Omega}.
$$
We denote by $\Gamma$ the point of the (potential) crack, and  by $\theta(\llbracket u \rrbracket)_\Gamma$ the value of the cohesive energy  $\theta$ on the crack amplitude of the displacement $\llbracket u \rrbracket$ on $\Gamma$. Since we are in a quasi-static setting, we introduce the time discretization $0=t_0<t_1< \cdots <t_T=T$ and look for the equilibrium configurations which are minimizers of the energy of the system. This means that  for each $i \in \{0, \cdots, T\}$ we need to minimize the energy of the system
$$
J(u)=\frac{1}{2}\int_{\Omega\backslash \Gamma}|a(x)\nabla u|^2 dx +\theta(\llbracket u \rrbracket)_\Gamma
$$ 
with respect to a given boundary datum $g$:
$$
u^*\in \argmin_{u=g(t_i) \mbox{ on } \partial \Omega} J(u).
$$
The function $a(\cdot)$ measures the degree of homogeneity of the material, e.g.  $a(x)\equiv 1$ means that the material is homogeneous.

In our experiments, we consider three different types of cohesive energy, the $\ell^\tau$ $\tau \in (0,1)$,  SCAD and  MCP penalties as defined in \eqref{p}, \eqref{SCAD}, \eqref{MCP} respectively.

In  paragraphs \ref{onedim} and \ref{twodim} we show our results for one-dimensional and two-dimensional experiments, respectively.
\subsubsection{One-dimensional experiments}\label{onedim}
We consider the one-dimensional domain $\Omega=[0,1]$  and   we chose the point of crack  as the midpoint $\Gamma=0.5$. 
 We divide $\Omega$ into $2N$ intervals  and  approximate the displacement function   with a function $u_h$ that is piecewise linear on $\Omega\backslash \Gamma$ and has two degrees of freedom on $\Gamma$ to represent correctly the two lips of the fracture, denoting with $u_{N}^{-}$ the one on $[0,0.5]$ and $u_{N}^{+}$ the one on $[0.5,1]$.  We discretize the problem in the following way
\begin{equation}\label{fractdiscr}
J_h(u_h)=\frac{1}{2} \sum_{i=1}^{2N} 2N |a_i(u_i -u_{i-1})|^2+ \theta(\llbracket u_N \rrbracket),
\end{equation}
where if $i\leq N$ we identify $u_N=u_N^-$ while for $i>N, u_N=u_N^+$ and $a_h$ denotes the piecewise linear approximation of the material inhomogeneity function.  We remark that the jump of the displacement is not taken into account in the sum, and the gradient of $u$ is approximated with finite difference of first order.
 The Dirichlet condition is applied on $\partial \Omega=\{0,2l\}$ and the external displacememt is chosen as
 $$
 u(0,t)=0, \quad u(2l,t)=2lt.
 $$
 To enforce the boundary condition in the minimization process, we add it to the energy functional as a penalization term. Hence, we solve the following  unconstrained minimization problem
 \begin{equation}\label{ff1}
 \min N|A u_h -b|_2^2+\theta(\llbracket u_N \rrbracket),
 \end{equation}
 where the operator $A \in \R^{(2N+1)\times (2N+1)}$ is given by $A=RD$ where $R\in \R^{(2N+1)\times (2N+1)}$ is the diagonal operator with $i$-entries  $R_{ii}=a_i$
 and
 $$
A=\left[\begin{array}{c}
\bar D\\
0\, \,\cdots\,\, 0 \, \,\gamma\end{array} \right].
$$
Here $\bar D \in \R^{2N\times (2N+1)}$  denotes the backward finite difference operator  $D	\, : \,\R^{2N+1} \to \R^{2N+1}$
 without the $N+1$ row, where $D$ is defined as
 \begin{equation}\label{Dcontrol}
D=\left(\begin{array}{ccccc}
1& 0& 0& \cdots& 0\\ -1& 1& 0& \cdots& 0\\\vspace{0.2cm}\\ 0& \cdots& 0&-1&1\
\end{array}\right).
\end{equation}
  Moreover
$b \in \R^{2N+1}$ in \eqref{ff1} is given by $b=(0,\cdots, \gamma t_i)'$ and $\gamma$ is the penalization parameter. To compute the jump between the two lips of the fracture, we introduce the operator $D_f:\R^{2N+1} \to \R$ defined as $D_f=(0,\cdots, -1,1,0,\cdots,0)$  where $-1$ and $1$ are respectively in the $N$ and $N+1$ positions.
Then we write the functional \eqref{ff1} as follows
 \begin{equation}\label{ff2}
 \min N|A u_h -b|_2^2+ \theta( D_fu),
 \end{equation}
We consider the three different  penalizations given by the $\ell^\tau, \tau \in (0,1)$, the SCAD and the MCP penalties. Note that    $\mbox{KerA }=0$, hence assumptions $(a) (ii)$ and $(c)(ii)$ are satisfied and existence  of a minimizer for \eqref{ff2} is  guaranteed.

Our numerical experiments were conducted with a discretization in $2N$ intervals with $N=100$. The time step in the time discretization of $[0,T]$ with $T=3$ is set to $dt=0.01$. The parameters of the energy functional $J_h(u_h)$ are set to $\lambda=1, \gamma=50$. 

We remark  that in the following experiments the material function $a(x)$ was always chosen as the identity. For tests with more general $a(x)$, we refer to the two-dimensional experiments reported in the following subsection.
In  Figures $1,2$ we report our results obtained by \textbf{Algorithm $1$} respectively for the  models $\ell^p$ and SCAD.  In each figure we show  time frames to represent the evolutions of the crack for different values of the parameter $\tau$. Each time frame consists of  three different time steps $(t_1, t_2, t_3)$, where $t_2, t_3$ are chosen as the first instant where the prefacture and the fracture appear. 

We observe  the three phases that we expect from a cohesive fracture model:
\begin{itemize}
\item \textit{Pure elastic deformation}: in this case the jump amplitude is zero and the gradient of the displacement is constant in $\Omega \backslash \Gamma$;
\item \textit{Prefracture}: the two lips of the fracture do not touch each other, but they are not free to move. The elastic energy is still present.
\item \textit{Fracture}: the two parts are free to move. In this final phase the gradient of the displacement (and then the elastic energy) is zero.
\end{itemize}
We remark that the  formation  of the crack is anticipated for smaller values of $\tau$. As we see in Figure $1$, for $\tau=.01$  prefracture and fracture  are reached at $t=.3$ and $t=1.5$ respectively. As $\tau$ is increased to $\tau=.1$,  prefracture and fracture occur at $t=1$ and $t=3$ respectively. We observe the same phenomenon for the SCAD (see Figure $2$).

We tested our algorithm also for the MCP model, where no prefracture phase can be observed, that is,   the displacement breaks almost instantaneously  to reach the complete fracture. 


Finally we remark that in our  experiments the residue is  $O(10^{-16})$ and the number of iterations is small, e.g. $12, 15$ for $\tau=.01, .1$ respectively.


\subsubsection{Two-dimensional experiments}\label{twodim}
We consider the two-dimensional domain $\Omega=(0,1)\times (0,1)$ and   we chose  the one-dimensional subdomain $0.5\times (0,1)$ as the line of crack.
 We proceed in the discretization of the problem analogously as in subsection \ref{onedim}, that is, we divide $[0,1]$ into $2N$ intervals  and  approximate the displacement function with a function $u_h$ that is piecewise linear in $\Omega\backslash \Gamma$ and has two degrees of freedom on $\Gamma$ to represent correctly the two lips of the fracture. 
 Define the   operator $A \in \R^{(2N+1)\times (2N+1)}$  by 
 $$
A=\left[\begin{array}{c}
R^1D_1\\
R^2D_2\\
\gamma I_{m^2}\end{array} \right],
$$
  where $m=2N+2$, $R^1, R^2 \in \R^{(2N+1)\times (2N+1)}$ are two diagonal operators approximating the degree of homogeneity of the material, $D_1\in \R^{(m-1)(m-2)\times m^2}, D_2\in \R^{m(m-2)\times m^2}$ are defined as
$$
D_1=G_1((m-1)N+1: (m-1)N+2(m-1),:)=[\quad], \quad D_2=G_2(mN+1:mN+m,:)=[\quad ],
$$
where
$G_1, G_2 \in \R^{m(m-1)\times m^2}$ are defined as follows
$$
G_1=kron(I_m,D), \quad G_2=kron(D,I_m)
$$ and 
$D \in \R^{m-1\times m}$  denotes the following backward finite difference operator  
 \begin{equation}\label{Dcontrol}
D=\left(\begin{array}{cccccc}
-1& 1& 0& \cdots & \cdots& 0\\0 &-1& 1& 0& \cdots& 0\\\vspace{0.2cm}\\ 0& \cdots& \cdots& 0&-1&1\
\end{array}\right).
\end{equation}
 Again we  enforce the boundary condition by adding it to the energy functional as a penalization term. 
Hence, we solve the following  unconstrained minimization problem
 \begin{equation}\label{ff12dim}
 \min |A u_h -b|_2^2+\theta(D_f u),
 \end{equation}
where $b \in \R^{(m-2)(2m-1)+m^2}$ in \ref{ff12dim} is given by $b=(0,\cdots, \gamma g(t_i))'$, $g(t_i)$ is the discretization of the boundary datum $g$ at time $t_i$ and $\gamma$ is the penalization parameter.
Moreover the jump of  the crack is represented by the operator $D_f\in \R^{m\times m^2}$ defined as follows
$$
D_f=[0_{m,mN}, -I_m, I_m, 0_{m,m^2-mN-2m}]
$$
where by $0_{r,s}$ we denote the null matrix of dimension $r\times s$.

Our numerical experiments were conducted with a discretization in $2N$ intervals with $N=80$. The time step in the time discretization of $[0,T]$ with $T=3$ is set to $dt=0.01$. The parameters of the energy functional $J_h(u_h)$ are set to $\lambda=1, \gamma=50$.  We perform two different series of experiments with boundary data respectively resulting from evaluating $g_1, g_2$ on $\partial \Omega$, where
$$
g_1(t)(x)=(2x_1-0.5)t \quad \mbox{ for every } t \in [0,1], x=(x_1,x_2) \in \Omega
$$
and the other one with boundary datum
$$
g_2(t)(x)=2t \cos(4(x_2-0.5))(x_1-0.5)\quad \mbox{ for every } t \in [0,1], x=(x_1,x_2) \in \Omega.
$$
In Figures $3,4,5,6$ we show the results obtained with boundary datum $g_1$ for each of the considered models, that is, $\ell^\tau$, SCAD and MCP and in Figures $7$ the ones with boundary datum $g_2$ for the $\ell^\tau$ model. In the case of boundary datum $g_2$ we tested our algorithm also on the SCAD and the MCP models, obtaining similar results to the ones shown in Figure $7$. In these first experiments, the diagonal operators $R_1, R_2$ are taken as the identity, that is, we suppose to have an homogeneous material.

As expected from a cohesive fracture model, we observe the three  phases of pure elastic deformation, prefracture and fracture. 

Also,  prefracture and fracture are  reached at different times for different values of $\tau$, typically they are anticipated for smaller values of $\tau$. 

When the boundary datum is $g_2$, that is, not constant in the $y$ direction, we note that the fracture is reached before in the part of the fracture line corresponding to the part of the boundary where the datum is bigger.

In Figures $8,9,10$ we tested the algorithm in case of a non homogeneous material. In Figure $8$ we show the result for  a two-material specimen, that is, we took
 \begin{equation}\label{R1}
\left\{
\begin{array}{lll}
R^1_{ii}=600 & i=1, \cdots, (m-1)(N+1), \\
R^1_{ii}=1 & i=(m-1)(N+1)+1, \cdots, 2(m-1)(N+1)
\end{array}
\right.\,
\end{equation}
\begin{equation}\label{R2}
\left\{
\begin{array}{lll}
R^2_{ii}=600 & i=1, \cdots, mN, \\
R^2_{ii}=1 & i=mN+1, \cdots, 2mN
\end{array}
\right.\,
\end{equation}
Note that, for the above values of $R^1, R^2$, the slides of the specimen show an asymmetric behaviour, namely  the displacement is flatter where the material function is bigger (that is, when $R_{ii}(x)=600$).
 
In Figure $9,10$ we report the results when $R^1,R^2$ are the discretization of the following function
\begin{equation}\label{Rr}
\left\{
\begin{array}{lll}
r(x,y)=400exp(y), & \mbox{ for } x\leq N \\
r(x,y)=400y & \mbox{ otherwise}
\end{array}
\right.\,
\end{equation}
Note that in Figure $10$ the boundary datum is chosen as
$$
g_3(t)=\frac{1}{100}cos(2(y-0.5))(x-0.5).
$$
As expected due to the choice of $R^1, R^2$, we remark an asymmetric behaviour of the fracture in the $y$ direction, namely the specimen brakes before where the material function is higher.

\subsection{Comparison with GIST}\label{compg}
In this section we present the result of experiments to compare the performance of \textbf{Algorithm 1} with  the following two other algorithms for nonconvex and nonsmooth minimization.
We first compare with the GIST "General Iterative Shrinkage and Thresholding" algorithm for $\ell^\tau$, $\tau<1$ minimization. We took advantage of the fact that for GIST\footnote{The reference paper is \cite{tuia}, the toolbox can be found in https://github.com/rflamary/nonconvex-optimization. } an open source toolbox is available, which facilitated an unbiased comparison.  Moreover, in \cite{LLSZ}  several tests were made to compare  GIST and  IRLS "Iteratively reweighted least squares", showing that  the two algorithms have nearly the same performance, with only significant difference in speed, where  GIST appears to be the faster one. 

Concerning $\ell^1$-minimization based algorithms, we compared our algorithm with the FISTA "Fast Iterative Shrinkage-Thresholding Algorithm", see subsection \ref{compg}.

We  remark that the results of \cite{LLSZ} show  no particular differences in the performance of the algorithm for different values of $\tau$, except that the speed becomes much worse for p near to $1$, say $\tau=0.9$.  Motivated also by this observations, the comparisons explained in the following were made for one fixed value of $\tau$.

The comparison is carried out through the following three examples, the academical M-matrix problem, an optimal control problem and a microscopy imaging reconstruction example. 

The monotone algorithm is stopped when the $\ell^\infty$-residue of the optimality condition \ref{optcondeps2m} is of the order of $10^{-3}$ in the $M$-matrix and optimal control problems and of the order of $10^{-8}$ in the imaging example. GIST is terminated if the relative change of the two consecutive objective function values is less than $10^{-5}$ or the number of iterations exceeds $1000$. We remark that no significant changes were remarked by setting a lower tolerance than $10^{-5}$ or a bigger number of maximal iteration for  GIST.

Since both  GIST and the FISTA solve the problem \eqref{optprobphi}  when the operator $\Lambda$ coincides with the identity, we also make this choice in the following subsections. Finally we remark that the three examples were analysed already in \cite{GK}  with different aims.

\subsubsection{M-matrix example}
We consider
\begin{equation}
\label{optprobM}
\min_{x \in \R^{n\times n}}J(x)= \min_{x \in \R^{n\times n}}\frac{1}{2}|A x-b|_2^2+\lambda |x|^\tau_\tau,
\end{equation}
$A$ is the forward finite difference gradient
$$
A=\left(\begin{array}{c} G_1\\G2\end{array}\right),
$$
with $G_1 \in \R^{n(n+1)\times n^2}, G_2 \in \R^{n(n+1)\times n^2}$ as
$$
G_1=I \otimes D, \quad G_2=D \otimes I,
$$
$I$ is the $n\times n$ identity matrix, $\otimes$ the tensor product, $D=(n+1)\tilde{D},$ $ \tilde{D} \in \R^{(n+1)\times n}$ is 
$$
\left(\begin{array}{ccccc}
1& 0& 0& \cdots& 0\\ -1& 1& 0& \cdots& 0\\ \vspace{0.2cm}\\ 0& \cdots& 0&-1&1\\0&\cdots&0&0&-1 
\end{array}\right).
$$
Then $A^T A$ is an $M$ matrix coinciding with the $5$-point star discretization on a uniform mesh on a square of the Laplacian with Dirichlet boundary conditions.  Moreover \eqref{optprobM} can be equivalently expressed as
\begin{equation}
\label{optprob2}
\min_{x \in \R^{n\times n}}\frac{1}{2}|A x|_2^2-(x,f)+\lambda |x|^\tau_\tau,
\end{equation}
where $f=A^T b$. If $\lambda=0$ this is the discretized variational form of the elliptic equation
\begin{equation}
\label{elleq}
-\Delta y=f \mbox{ in } \Omega, \quad y=0 \mbox{ on } \partial \Omega.
\end{equation}
For $\lambda>0$ the variational problem \eqref{optprob2} gives a sparsity enhancing solution for the elliptic equation \eqref{elleq}, that is, the displacement $y$ will be $0$ when the forcing $f$ is small.
Our tests are conducted with $f$ chosen as discretization of $f=10 x_1\mbox{sin}(5x_2) \mbox{cos}(7 x_1)$. The inizialization is chosen as the solution of the corresponding non-sparse optimization problem.\\
We  remark that in \cite{GKproc} and  \cite{GK} the algorithm was also tested   in the same situation for different values of $\tau$ and $\lambda$, showing,   in particular and consistent with our expectations, that the sparsity of the solution increases with $\lambda$. 

Here we focus on the comparison between the performances of \textbf{Algorithm 1} and GIST. In order to compare the two schemes, we focus on the value of the unregularized functional $J$ in \eqref{optprobM} reached by both algorithms, the time to acquire it, and the number of iterations.  Our tests were conducted   for $\tau=0.5$, and $\lambda$ incrementally increasing from $10^{-3}$ to $0.3$, see the following tables.  The parameter $\eps$ was decreased from $10^{-1}$ to $10^{-6}$.    These values are reported in Table  $1,2$ and $3,4$ for  GIST and \textbf{Algorithm 1} respectively. 

We observe that \textbf{Algorithm 1} achieves always lower values of the functional J, but in a longer time. The number of iterations needed by \textbf{Algorithm 1} is smaller than the number of iterations of  GIST for small values of $\lambda$, more precisely for $\lambda <0.1$. Note that  for smaller $\lambda$ the number of iterations of \textbf{Algorithm 1} is smaller than the one of  GIST.  This suggests, consistent with our expectation, that the monotone scheme is slower than  GIST mainly because it solves a nonlinear equation at each iteration. 

We carried out a further test in order to measure the timing performance of \textbf{Algorithm 1}, that is, the algorithm is stopped as soon  as the value of J achieved by GIST is reached. 
In Table $4,5$ we report the time, the number of iterations, the values of J, and the value of $\eps$ reached. We observe that the time is almost always smaller than the one of  GIST, except for values of $\lambda$ bigger or equal than $\lambda=0.15$. Also, for these values, the differences in the  time are very small.

\begin{table}[h!]	
	
\begin{center}
\begin{tabular}{|c|c|c|c|c|c|c|}
	\hline\noalign{\smallskip}
		 {\bf $\lambda$ } & 0.01& 0.05 &  0.10&0.15&0.2 & 0.3  \\
		\noalign{\smallskip}
		\hline
		\noalign{\smallskip}
	J & 246.324 & 264.232 & 285.26& 303.685 & 319.737 &338.998\\
		 time & 0.563 & 0.701  & 0.444 &0.468 & 0.461 & 0.61\\
		 iterations & 293& 384  & 249&247 & 216 & 209 \\
		\noalign{\smallskip}\hline

		 \end{tabular}
		 \caption{M-matrix example. Value of J, time,  iterations of  GIST.}
		 \end{center}
\end{table}

\begin{table}[h!]
		 \begin{center}
		\begin{tabular}{|c|c|c|c|c|c|c|}
	\hline\noalign{\smallskip}
		 {\bf $\lambda$ } & 0.01& 0.05 &  0.10&0.15&0.2 & 0.3  \\
		 \noalign{\smallskip}
		 \hline
		 \noalign{\smallskip}
		 J & 246.186 & 263.92  & 284.079&  301.327 & 315.553 & 331.71 \\
		 time & 10.92 & 26.142 & 56.397& 33.021 & 124.624 & 31.423 \\
		 iterations & 149 & 361 & 779 & 456 & 1722 & 433\\
		 \noalign{\smallskip}\hline
		 \end{tabular}
		 \caption{M-matrix example. Value of J, time,  iterations of \textbf{Algorithm 1}.}
		 \end{center}
\end{table}

\begin{table}[h!]
\begin{center}
\begin{tabular}{|c|c|c|c|c|c|c|c|} 
\hline\noalign{\smallskip}
{\bf $\lambda$}& 0.001 & 0.01& 0.05 &  0.1&0.15&0.2  \\
		 \noalign{\smallskip}
		 \hline
		 \noalign{\smallskip}
		 J$_{\mbox{\footnotesize{GIST}}}$ &242.158 & 246.324 & 264.232 & 285.26 & 303.685 & 319.737 \\
		 iter$_{\mbox{\footnotesize{mon}}}$ & 1 & 1 & 5  &5 & 6 & 7 \\
		 time$_{\mbox{\footnotesize{mon}}}$ & 0.085 & 0.082& 0.39&  0.387 &0.478& 0.673 \\
		 time$_{\mbox{\footnotesize{GIST}}}$ & 0.445 & 0.563 & 0.701  & 0.444  & 0.468 &0.461  \\
		\noalign{\smallskip}
		\hline
		 \end{tabular}
		 \caption{M-matrix example. Value of the functional, iterations, time to which \textbf{Algorithm 1} overcome  GIST's.}
		 \end{center}
\end{table}

\subsubsection{Optimal control problem}
We consider the linear control system
$$
\frac{d}{dt} y(t)=\mathcal{A} y(t)+B u(t), \quad y(0)=0,
$$
that is,
\begin{equation}\label{LCSfinalstate}
y(T)=\int_0^T e^{\mathcal{A}(T-s)} B u(s) ds,
\end{equation}
where the linear closed operator $\mathcal{A}$ generates a $C_0$-semigroup $e^{\mathcal{A}t}$, $t\geq 0$ on the Hilbert space $X$. More specifically, we consider the  one-dimensional controlled heat equation for $y=y(t,x)$:
\begin{equation}\label{actionscontrol}
y_t=y_{xx}+b_1(x)u_1(t)+b_2(x)u_2(t), \quad x \in (0,1),
\end{equation}
with homogeneous boundary conditions $y(t,0)=y(t,1)=0$ and thus $X=L^2(0,1)$. The differential operator $\mathcal{A}y=y_{xx}$ is discretized in space by the second order finite difference approximation with $n=49$ interior spatial nodes ($\Delta x=\frac{1}{50}$). We use two time dependent controls $\overrightarrow u=(u_1,u_2)$ with corresponding spatial control distributions $b_i$ chosen as step functions:
$$
b_1(x)=\chi_{(.2,.3)}, \quad b_2(x)=\chi_{(.6,.7)}.
$$
The control problem consists in finding the control function $\overrightarrow u$ that steers the state $y(0)=0$ to a neighbourhood of the desired state $y_d$ at the terminal time $T=1$. We discretize the problem in time by the mid-point rule, i.e.
\begin{equation}\label{midpoint}
A \overrightarrow  u=\sum_{k=1}^m e^{\mathcal{A}\left(T-t_{k}-\frac{\Delta t}{2}\right)} (B \overrightarrow u)_k \Delta t,
\end{equation}
where $\overrightarrow u=(u_1^1,\cdots, u_1^m,u_2^1,\cdots u_2^m)$ is a discretized control vector whose coordinates represent the values at the mid-point of the intervals $(t_k,t_{k+1})$. Note that in \eqref{midpoint} we denote by $B$ a suitable rearrangement of the matrix $B$ in \eqref{LCSfinalstate} with some abuse of notation. A uniform step-size $\Delta t=\frac{1}{50}$ ($m=50$) is utilized. The solution of the control problem is based on the sparsity formulation \eqref{optprobphi}, where  $\Lambda=I$ and $\phi_{\lambda, \tau}(x)=\lambda |x|^{\tau}$ and 
 $b$ in \eqref{optprobphi} is the discretized target function chosen as the Gaussian distribution $y_d(x)=0.4\,\mbox{exp}(-70(x-.7)^2))$ centered at $x=.7$.
 That is, we apply our algorithm for the discretized optimal control problem in time and space where $x$ from \eqref{optprobphi} is the discretized control vector $u \in \R^{2m}$ which is mapped by $A$ to the discretized output $y$ at time $1$ by means of \eqref{midpoint}. Moreover $b$ from \eqref{optprobphi} is the discretized state $y_d$ with respect to the spatial grid $\Delta x$. The parameter $\eps$ was initialized with $10^{-3}$ and decreased down to $10^{-8}$.

Similarly as in the previous subsection, we compare the values of the functional, the time and the number of iterations. The experiments are carried out for $\tau=0.5$ and $\lambda$ in the interval $10^{-3}$-$0.2$. We report only the values for the second control $u_2$ since the first control $u_1$ is always zero (as expected).

As  can be  seen from the following tables, the same kind of remarks as in the previous subsection apply. In particular  GIST is  faster but less precise than \textbf{Algorithm 1}, but \textbf{Algorithm 1} overcomes the value reached by  GIST more rapidly.

	\begin{table}[h!]	
\begin{center}
\begin{tabular}{|c|c|c|c|c|c|c|c|c|} 
\hline\noalign{\smallskip}
		   $\lambda$& 0.0001 & 0.001& 0.01 &  0.2\\
		 \noalign{\smallskip}
		 \hline
		 \noalign{\smallskip}
		 J & 0.044 & 0.073 & 0.599 & 0.599\\
		 time & 0.296 & 0.047 & 0.04  & 0.037 \\
		 iterations  & 222& 157& 3& 3  \\
		 \noalign{\smallskip}
		\hline
		 \end{tabular}
		 \caption{Optimal control problem. Value of J, time, iterations of  GIST. }
		 \end{center}
\end{table}

\begin{table}[h!]
		 \begin{center}
		 \begin{tabular}{|c|c|c|c|c|c|c|c|c|}
		 \hline\noalign{\smallskip}
		 $\lambda$ & 0.0001& 0.001&0.01 & 0.2 \\
		\noalign{\smallskip}
		 \hline
		 \noalign{\smallskip}
		 J  & 0.042 &  0.068& 0.185 & 0.599\\
		 time&  11.758 & 15.140 & 14.866 & 12.501 \\
		 iterations  & 35 & 28 & 32 & 27\\
		 \noalign{\smallskip}
		\hline
		 \end{tabular}
		 \caption{Optimal control problem. Value of J, time, iterations of \textbf{Algorithm 1}.}
		 \end{center}
\end{table}

\begin{table}[h!]
\begin{center}
\begin{tabular}{|c|c|c|c|c|c|c|c|c|} 
\hline\noalign{\smallskip}
		   $\lambda$& 0.0001 & 0.001& 0.01 &  0.2\\
		\noalign{\smallskip}
		 \hline
		 \noalign{\smallskip}
		 J$_{\mbox{\footnotesize{mon}}}$ & 0.043 & 0.071 & 0.185  &  0.599 \\
		 iter$_{\mbox{\footnotesize{mon}}}$ & 1 & 1 & 5  &5 \\
		  time$_{\mbox{\footnotesize{mon}}}$ & 2.2 & 0.1& 0.39&  0.025 \\
		  time$_{\mbox{\footnotesize{GIST}}}$ & 0.296 & 0.047 & 0.04  & 0.037 \\
		 \noalign{\smallskip}
		\hline
		 \end{tabular}
		 \caption{Optimal control problem. Value of J, iterations, time for which \textbf{Algorithm 1} overcomes  GIST's.}
		 \end{center}
\end{table}

\subsubsection{Compressed sensing approach for microscopy  image reconstruction}\label{mimrec}
We compare \textbf{Algorithm 1} and  GIST in a microscopy imaging problem, in particular we focus on the STORM  (stochastic optical reconstruction microscopy) method,  based on stochastically switching and high-precision detection of single molecules  to achieve an image resolution beyond the diffraction limit.  The literature on the STORM has  been intensively increasing, see e.g. \cite{RBZ}, \cite{BPS} \cite{HGM}, \cite{HBZ}. We refer in particular to \cite{GK}   for a detailed description of the method and for more references.

Our approach is based on the following constrained-minimization problem: 
\begin{equation}\label{minprobimage}
 \min_{x \in \R^n} |x|^\tau_{\tau} \quad \mbox{ such that } \, \,|A x-b|_2 \leq \eps,
\end{equation}
where $\tau \in (0,1]$, $x$ is the up-sampled, reconstructed image, $b$ is the experimentally observed image, and $A$ is the impulse response (of size $m\times n$, where $m$ and $n$ are the numbers of pixels in $b$ and $x$, respectively). $A$ is usually called the  point spread function (PSF) and  describes the response of an imaging system to a point source or point object.   
Problem \eqref{minprobimage} can be reformulated as: 
\begin{equation}\label{minprobimagebeta}
\min_{x \in \R^n} \frac{1}{2}|Ax-b|^2_2+\lambda |x|^\tau_\tau.
\end{equation}

%
%
First we tested the procedure for same resolution images, in particular the conventional and the true images are both $128\times 128$ pixel images.  
 Then the algorithm was tested in the case of a $16\times 16$ pixel conventional image and a $128 \times 128$ true image. 
The values for the impulse response $A$ and the measured data $b$ were chosen according to the literature, in particular $A$ was taken as the Gaussian PSF matrix with variance $\sigma=8$ and  size $3\times \sigma=24$, and $b$ was simulated by convolving the impulse response $A$ with a random $0$-$1$ mask over the image  adding a white random noise so that the signal to noise ratio is  $.01$. 

 We carried out several tests with the same data for different values of  $\tau,\lambda$. We report only our results for  $\tau=.1$ and $\lambda=10^{-6}, \lambda=10^{-9}$ for the same and the  different resolution case respectively, since for these values the best reconstructions were achieved. We focus on two different type of images, a sparse $0$-$1$ cross-like image and the standard phantom image.
In order to compare the performance of \textbf{Algorithm 1} and the GIST algorithm, we focus on the number of surplus emitters (Error+) and missed emitters (Error-) recovered   in the case of the cross image and different resolution.  The errors are computed on an  average over six recoveries for different values of the noise. The graphics of the errors against the noise are reported in Figures $11$ and $12$  for \textbf{Algorithm 1} and GIST respectively. We remark that these quantities  are typically used as a measure of the efficacy of the reconstruction method, see for example \cite{DP} (where, under certain conditions, a linear decay with respect to the noise is proven) and \cite{C}. 

The results shows that by  GIST the Error- is always $197$, whereas by \textbf{Algorithm 1} is always under $53$ and even smaller for small values of the noise. On the other hand, the Error+ by  GIST is always 0 and by \textbf{Algorithm 1} is zero for small values of the noise and then monotonically increasing until it reaches  $175$ when the noise is equal to $0.1$. Consistently with what expected, by \textbf{Algorithm 1} the graphics show a linear decay w.r.t. the noise, differently from the behaviour showed by  GIST. Moreover, the results found by \textbf{Algorihtm 1} lead to  more accuracy in the recovery, in the sense that the quantity of missed emitters is smaller, whereas on the other hand  GIST seems to lead to a more sparser solutions (since the Error+ is 0 by GIST).

Finally we remark that in the case of the cross image  GIST is faster than our algorithm, consistently with the result presented in the previous subsection and as expected, since our algorithm solves a nonlinear equation for each minimization problem. On the other hand, in the case of the standard phantom image  GIST results to be far slower than \textbf{Algorithm 1}. 

In Figure $13$ we report the results obtained in the same situation by the FISTA "Fast Iterative Shrinkage Thresholding Algorithm" for $\ell^1$ minimization. 
 We remark that by the FISTA the Error+  is always above $400$, whereas by \textbf{Algorithm $1$} is zero for small value of the noise.  This shows that \textbf{Algorithm 1} leads to more sparsity with respect to the FISTA, consistently with our expectation since the FISTA is based on $\ell^1$ minimization. 

\section{Conclusions}
We have developed a  monotone convergent algorithm for a class of nonconvex nonsmooth optimization problems
arising in the modelling of fracture mechanics and in imaging reconstruction, including the $\ell^\tau, \tau \in (0,1]$, the smoothly clipped absolute deviation and the minimax concave penalty. Theoretically, we established the existence of a minimizer of the original problem under  assumptions implying coercivity of the functional. Then we derived  necessary optimality conditions for a regularized version of the original problem. The optimality conditions for the regularized problem were solved through a monotonically convergent scheme based on an iterative procedure. We proved the convergence of the iteration procedure  under the same assumptions that guarantee existence. A remarkable result is the strict monotonicity of the functional along the sequence of iterates generated by the scheme.  Moreover we proved the convergence of the regularized problem to the original one as the regularization parameter goes to zero.  

The procedure is very efficient and accurate. The efficiency and accuracy of the procedure was verified by numerical tests simulating the evolution of cohesive fractures  and microscopy imaging.
An issue of high relevance to us was the comparison of the scheme to  two alternative algorithms, the GIST "General Iterative Shrinkage and Thresholding" algorithm for $\ell^\tau$ minimization,  with $\tau$ strictly positive and less than $1$  and  the FISTA "Fast Iterative Shrinkage-Thresholding Algorithm"  for $\ell^1$ minimization. We first compared  with GIST by focusing on the infimal value reached by the iteration procedure and on the computing time. Our results showed that the monotone algorithm is able to reach a smaller value of the objective functional  when compared to GIST's, therefore leading to a better accuracy. Finally we compared our scheme with FISTA in sparse recovery related to  microscopy imaging. The results showed that the monotone scheme lead to  more sparsity with respect to FISTA, as expected since  FISTA concerns $\ell^1$ minimization.

\begin{figure}[!ht]
\centering
\subfloat[$\footnotesize{t=.2}$]{\includegraphics[height=4cm, width=1.7cm]{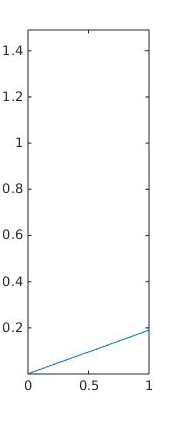}}
\subfloat[$\footnotesize{t=.3}$]{\includegraphics[height=4cm, width=1.7cm]{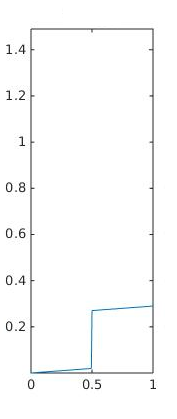}}
\subfloat[$\footnotesize{t=1.5}$]{\includegraphics[height=4cm, width=1.7cm]{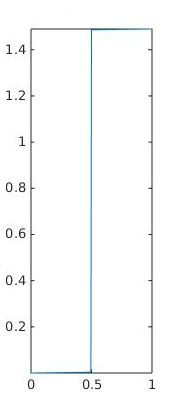}}
%
\hspace{0.5cm}
\subfloat[$\footnotesize{t=.9}$]
{
\includegraphics[height=4cm, width=1.7cm]{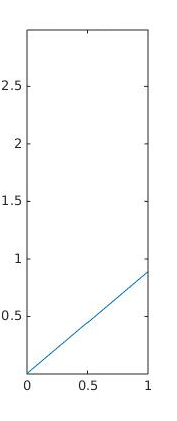}
}
\subfloat[$\footnotesize{t=1}$]
{
\includegraphics[height=4cm, width=1.7cm]{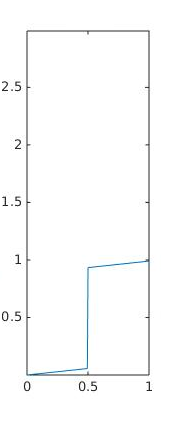}
}
\subfloat[$\footnotesize{t=3}$]
{
\includegraphics[height=4cm, width=1.7cm]{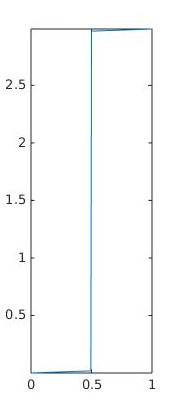}
}
\caption{Three time-step evolution of the displacement for  $\tau=.01$, $t = .2, .3, 1.5$  (left), $\tau=.1$, $t=.9, 1, 3$ (right). Results obtained by \textbf{Algorithm $1$}.}
\end{figure}
\vspace{0.3cm}

\begin{figure}[!ht]
\centering
\subfloat[$\footnotesize{t=1}$]{\includegraphics[height=4cm, width=1.7cm]{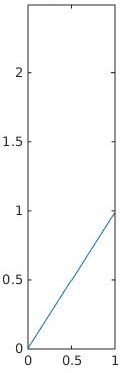}}
\subfloat[$\footnotesize{t=2.1}$]{\includegraphics[height=4cm, width=1.7cm]{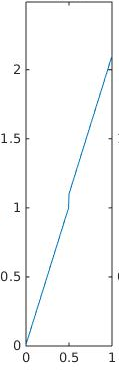}}
\subfloat[$\footnotesize{t=2.2}$]{\includegraphics[height=4cm, width=1.7cm]{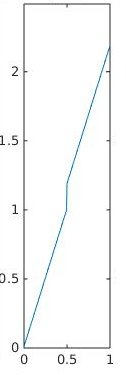}}
\subfloat[$\footnotesize{t=2.5}$]{\includegraphics[height=4cm, width=1.7cm]{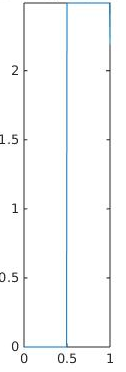}}
%
\hspace{0.4cm}
\subfloat[$\footnotesize{t=.1}$]{\includegraphics[height=4cm, width=1.7cm]{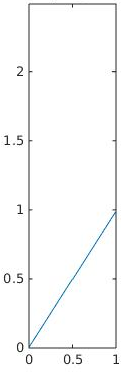}}
\subfloat[$\footnotesize{t=2.1}$]{\includegraphics[height=4cm, width=1.7cm]{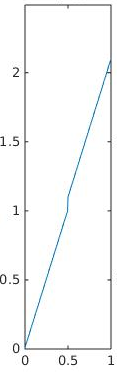}}
\subfloat[$\footnotesize{t=2.2}$]{\includegraphics[height=4cm, width=1.7cm]{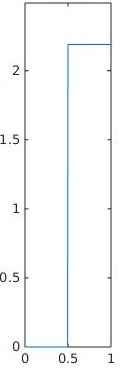}}
\subfloat[$\footnotesize{t=2.5}$]{\includegraphics[height=4cm, width=1.7cm]{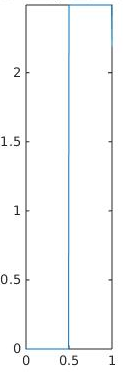}}
\caption{Four time-step evolution of the displacement for the SCAD model,  $\tau=20$,  $t=1, 2.1, 2.2, 2.5$   (left), $\tau=10$, $t = .1, 2.1, 2.2, 2.5$ (right). Results obtained by \textbf{Algorithm $1$}.}
\end{figure}

\vspace{1cm}

\begin{figure}[!ht]
\subfloat[$t=0.1, \tau=0.001$]
{
\includegraphics[height=4.5cm, width=4.5cm]{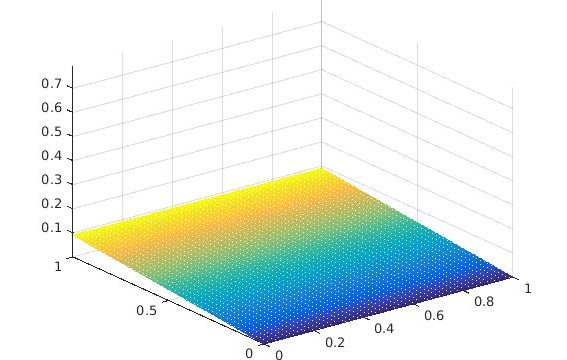}
}
\subfloat[$t=0.8, \tau=0.001$]
{
\includegraphics[height=4.5cm, width=4.5cm]{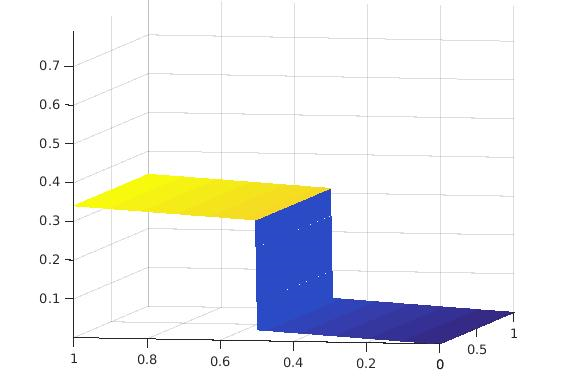}
}
\subfloat[$t=0.8, \tau=0.001$]
{
\includegraphics[height=4.5cm, width=4.5cm]{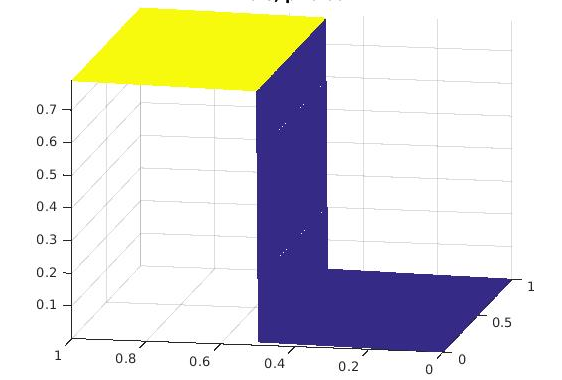}
}
\caption{Displacement, $\theta(\cdot)=|\cdot|_\tau^\tau$, with $\tau=0.001$, $R^1=R^2=I$, and boundary datum $g=g_1$}
\end{figure}

\vspace{2cm}

\begin{figure}[!ht]
\subfloat[$t=0.1, \tau=0.01$]
{
\includegraphics[height=4.5cm, width=4.5cm]{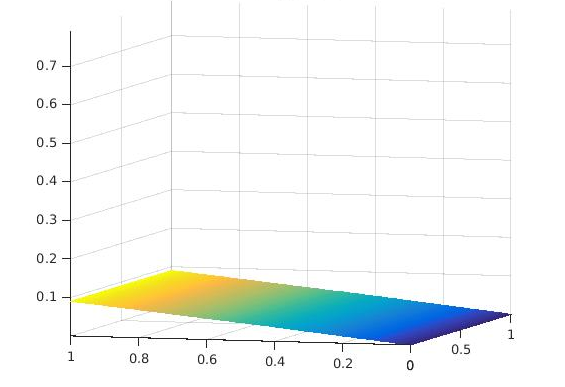}
}
\subfloat[$t=0.35, \tau=0.01$]
{
\includegraphics[height=4.5cm, width=4.5cm]{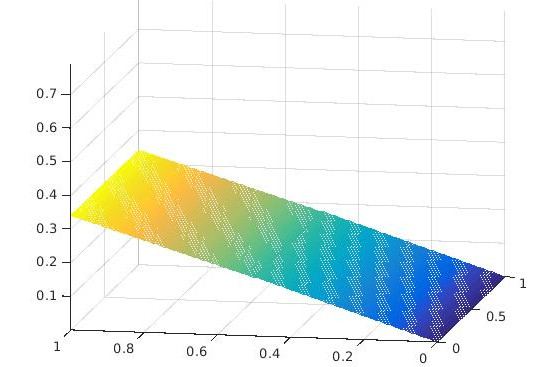}
}
\subfloat[$t=0.8, \tau=0.01$]
{
\includegraphics[height=4.5cm, width=4.5cm]{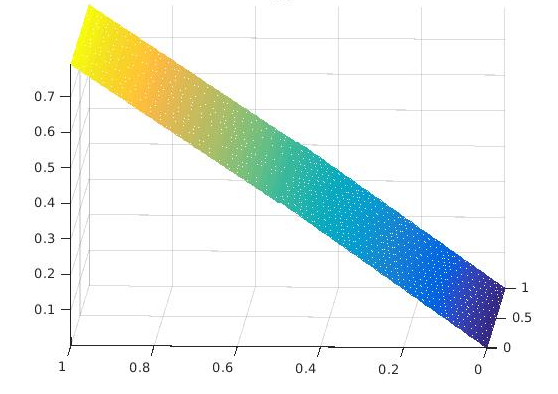}
}
\newline
\hspace{-5cm}
\subfloat[$t=0.1, \tau=0.0001$]
{
\includegraphics[height=4.5cm, width=4.5cm]{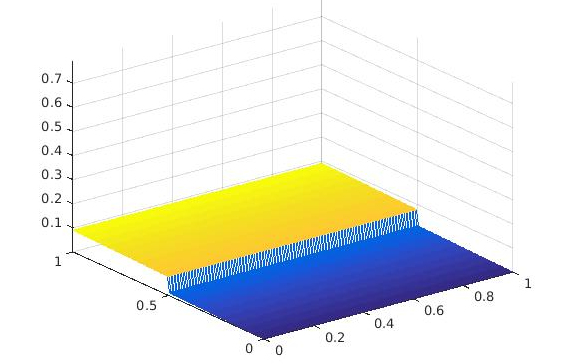}
}
\subfloat[$t=0.8, \tau=0.0001$]
{
\includegraphics[height=4.5cm, width=4.5cm]{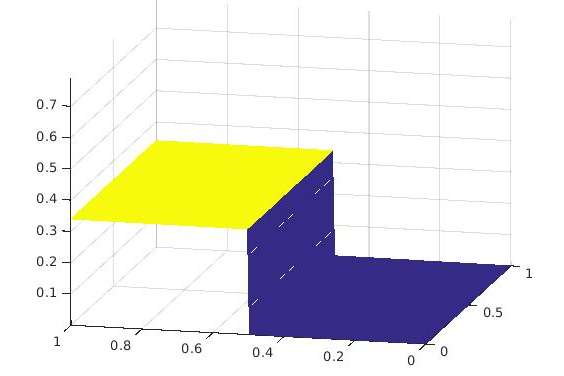}
}
\subfloat[$t=0.8, \tau=0.0001$]
{
\includegraphics[height=4.5cm, width=4.5cm]{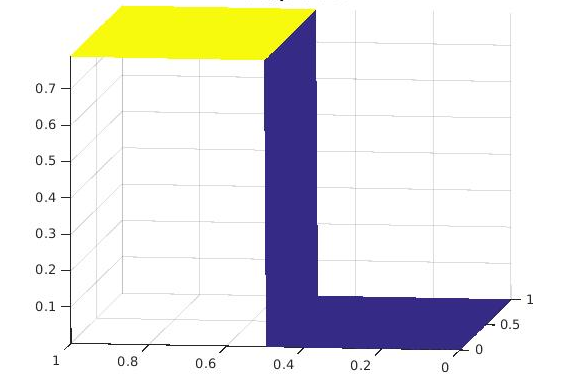}
}
\caption{Displacement, $\theta(\cdot)=|\cdot|_\tau^\tau$, comparison between $\tau=0.01$ and $\tau=0.0001$, $R^1=R^2=I$,  boundary datum $g=g_1$.}
\end{figure}

\vspace{1cm}

\begin{figure}[!ht]
\subfloat[$t=1.5, \tau=20$]
{
\includegraphics[height=3.5cm, width=3.5cm]{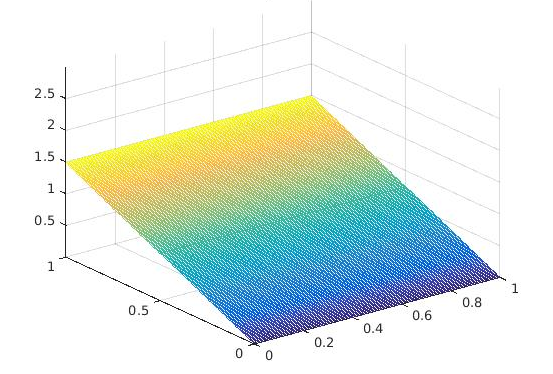}
}
\subfloat[$t=1.7, \tau=20$]
{
\includegraphics[height=3.5cm, width=3.5cm]{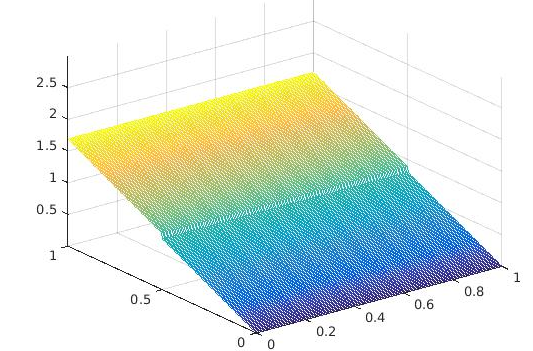}
}
\subfloat[$t=2.5, \tau=20$]
{
\includegraphics[height=3.5cm, width=3.5cm]{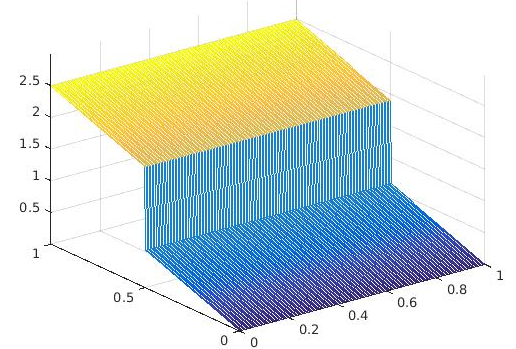}
}
\subfloat[$t=3, \tau=20$]
{
\includegraphics[height=3.5cm, width=3.5cm]{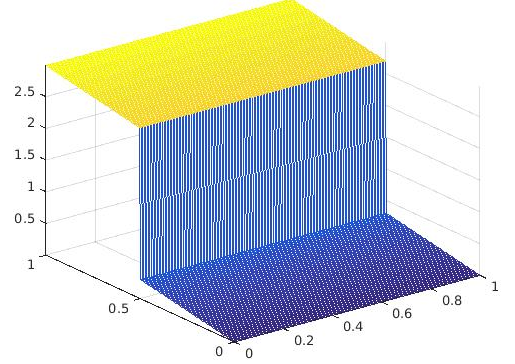}
}
\newline
\hspace{-5cm}
\subfloat[$t=1, \tau=15$]
{
\includegraphics[height=3.5cm, width=3.5cm]{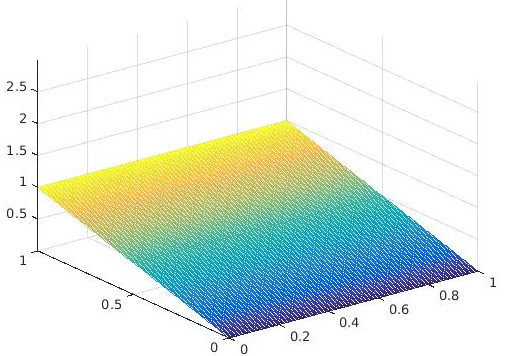}
}
\subfloat[$t=2.8, \tau=15$]
{
\includegraphics[height=3.5cm, width=3.5cm]{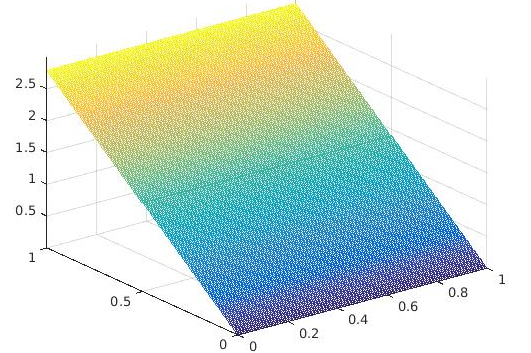}
}
\subfloat[$t=3, \tau=15$]
{
\includegraphics[height=3.5cm, width=3.5cm]{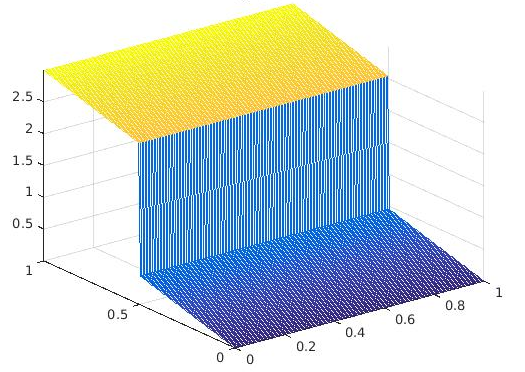}
}
\caption{Displacement for the $\mbox{SCAD}$ model, comparison between $\tau=20$ and $\tau=15$, $R^1=R^2=I$, boundary datum $g=g_1$.}
\end{figure}

\begin{figure}[!ht]
\subfloat[$t=1.5, \tau=20$]
{
\includegraphics[height=3.3cm, width=3.5cm]{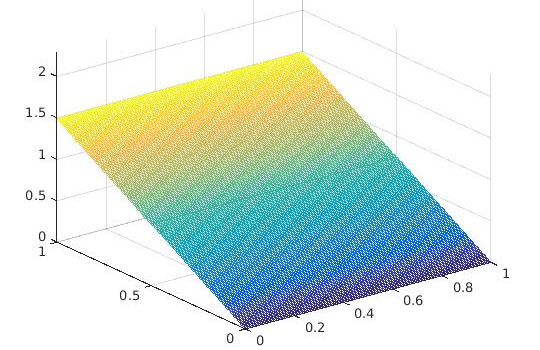}
}
\subfloat[$t=1.8, \tau=20$]
{
\includegraphics[height=3.3cm, width=3.5cm]{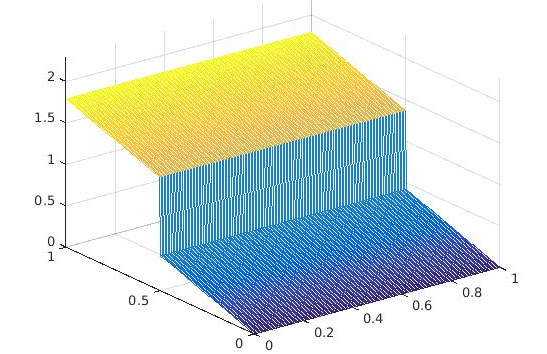}
}
\subfloat[$t=2.3, \tau=20$]
{
\includegraphics[height=3.3cm, width=3.5cm]{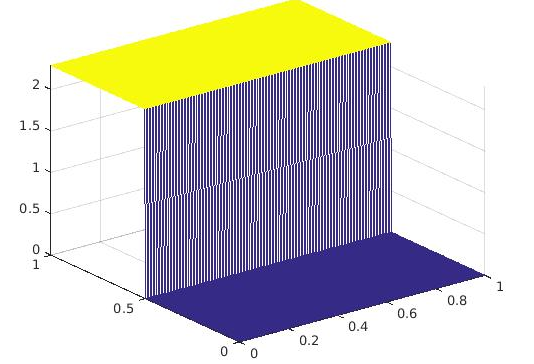}
}
\newline
\hspace{-5cm}

\subfloat[$t=1.5, \tau=10$]
{
\includegraphics[height=3.3cm, width=3.5cm]{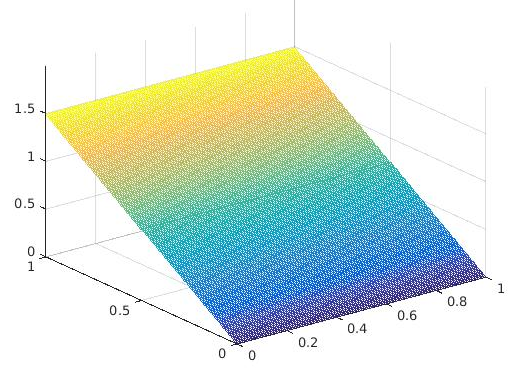}
}
\subfloat[$t=1.56, \tau=10$]
{
\includegraphics[height=3.3cm, width=3.5cm]{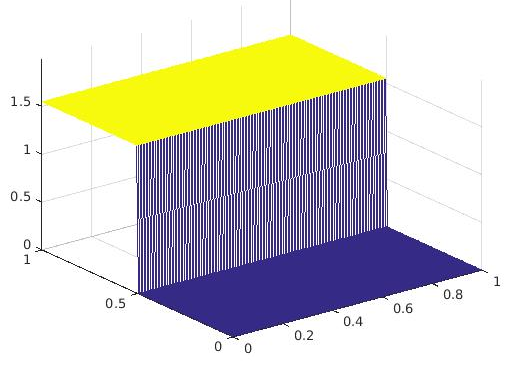}
}
\subfloat[$t=2, \tau=10$]
{
\includegraphics[height=3.3cm, width=3.5cm]{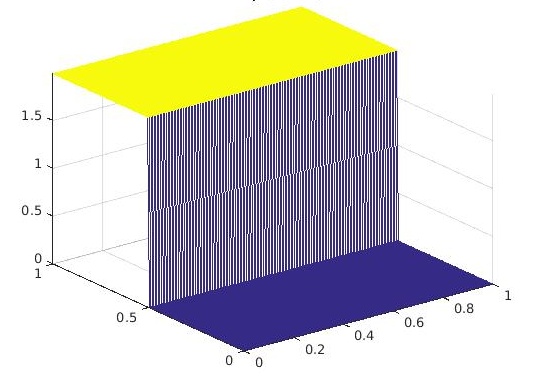}
}
\caption{Displacement for the MCP model, comparison between $\tau=20$ and $\tau=10$, $R^1=R^2=I$, boundary datum $g=g_1$.}
\end{figure}

\begin{figure}[!ht]

\subfloat[$t=0.4, \tau=0.001$]
{
\includegraphics[height=3.2cm, width=3.5cm]{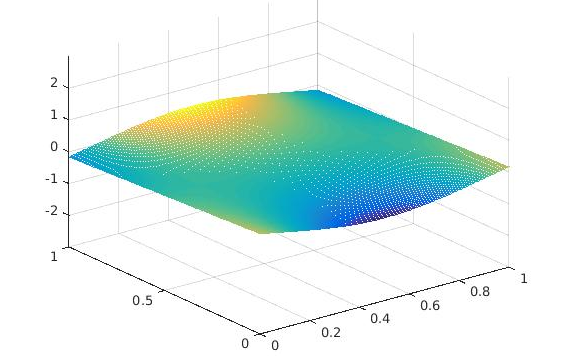}
}
\subfloat[$t=2, \tau=0.001$]
{
\includegraphics[height=3.2cm, width=3.5cm]{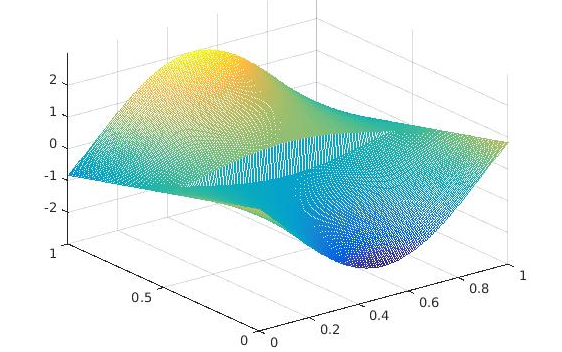}
}
\subfloat[$t=3, \tau=0.001$]
{
\includegraphics[height=3.2cm, width=3.5cm]{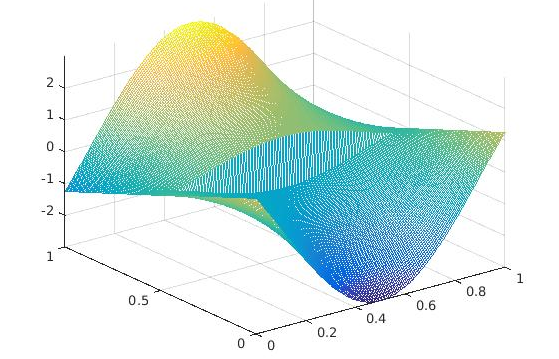}
}
\newline
\hspace{-5cm}
\subfloat[$t=0.4, \tau=0.0001$]
{
\includegraphics[height=3.2cm, width=3.5cm]{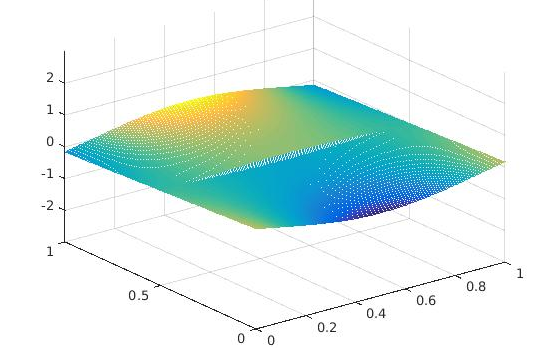}
}
\subfloat[$t=2, \tau=0.0001$]
{
\includegraphics[height=3.2cm, width=3.5cm]{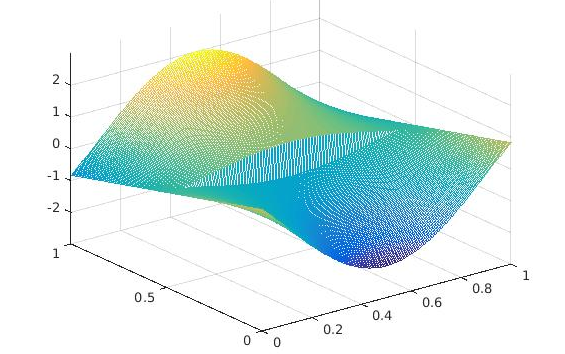}
}
\subfloat[$t=3, \tau=0.0001$]
{
\includegraphics[height=3.2cm, width=3.5cm]{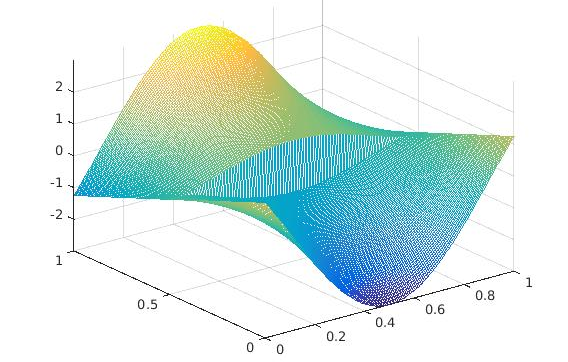}
}
\caption{Displacement, $\theta(\cdot)=|\cdot|_\tau^\tau$, comparison between $\tau=0.001$ and $\tau=0.0001$, $R^1=R^2=I$, boundary datum $g=g_2$.}
\end{figure}

%
%
%
%
%

\vspace{-0.4cm}

\begin{figure}[!ht]
\subfloat[$t=0.9, \tau=0.01$]
{
\includegraphics[height=3.2cm, width=3.5cm]{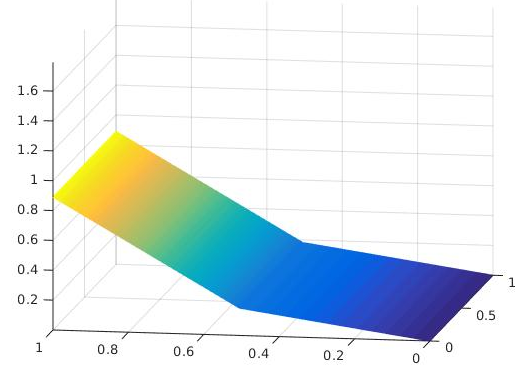}
}
\subfloat[$t=1.1, \tau=0.01$]
{
\includegraphics[height=3.2cm, width=3.5cm]{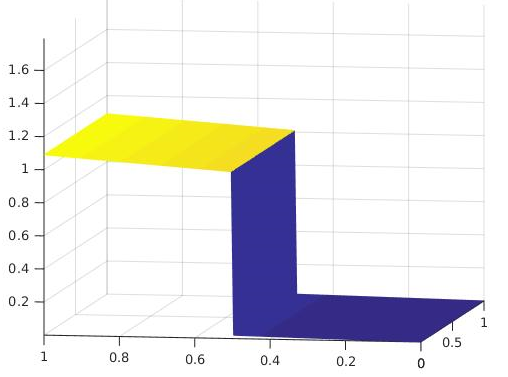}
}
\subfloat[$t=1.5, \tau=0.01$]
{
\includegraphics[height=3.2cm, width=3.5cm]{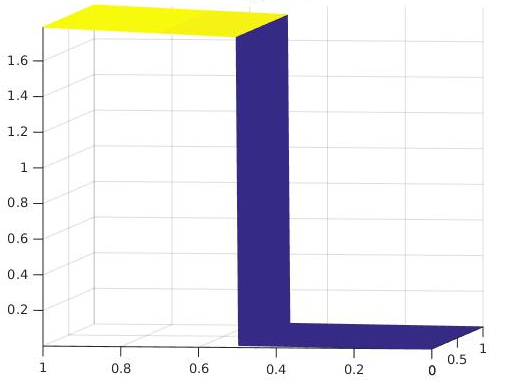}
}
\caption{Displacement, $\theta(\cdot)=|\cdot|_\tau^\tau$,  $\tau=0.01$, $R^1, R^2$ given by \eqref{R1}-\eqref{R2}, boundary datum $g=g_1$.}
\end{figure}

\vspace{-0.4cm}

\begin{figure}[!ht]
\subfloat[$t=1.5, \tau=0.1$]
{
\includegraphics[height=3.2cm, width=3.5cm]{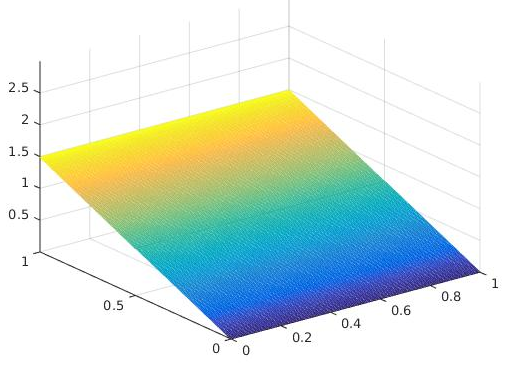}
}
\subfloat[$t=2, \tau=0.1$]
{
\includegraphics[height=3.2cm, width=3.5cm]{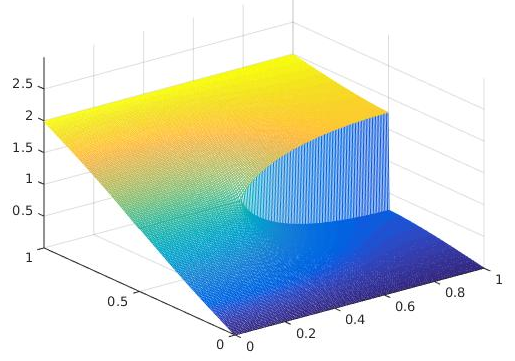}
}
\subfloat[$t=3, \tau=0.1$]
{
\includegraphics[height=3.2cm, width=3.5cm]{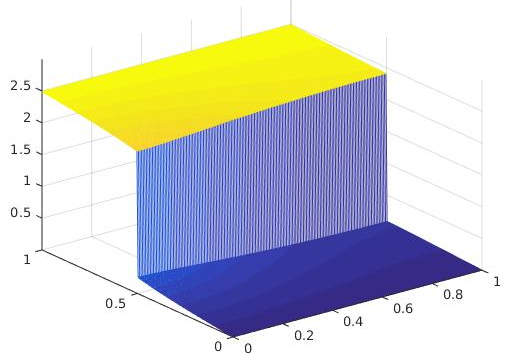}
}
\caption{Displacement, $\theta(\cdot)=|\cdot|_\tau^\tau$, $\tau=0.01, R^1,R^2$ given by  \eqref{Rr},  boundary datum $g=g_1$}
\end{figure}

\begin{figure}[!ht]
\subfloat[$t=0.2, \tau=0.1$]
{
\includegraphics[height=3.5cm, width=3.5cm]{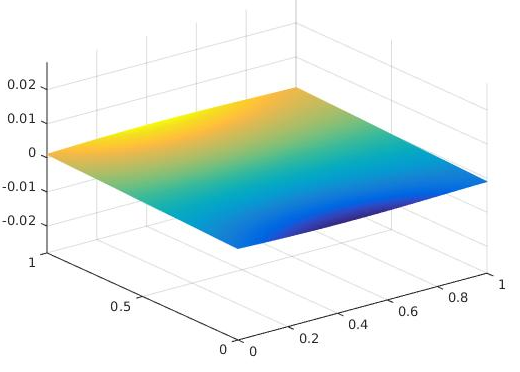}
}
\subfloat[$t=1.5, \tau=0.1$]
{
\includegraphics[height=3.5cm, width=3.5cm]{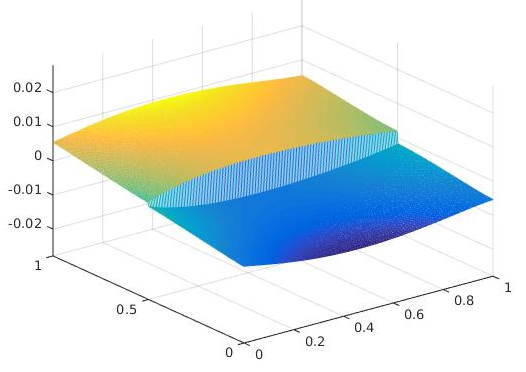}
}
\subfloat[$t=3, \tau=0.1$]
{
\includegraphics[height=3.5cm, width=3.5cm]{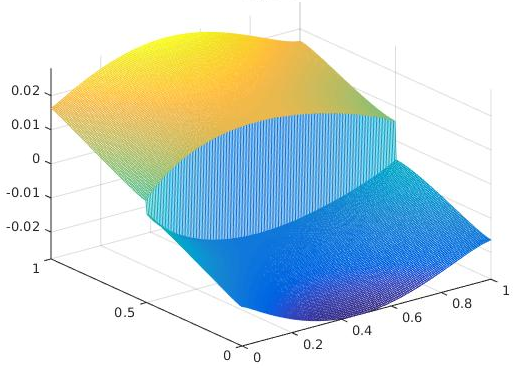}
}
\caption{Displacement, $\theta(\cdot)=|\cdot|_\tau^\tau$,  $\tau=0.1$, $R^1,R^2$ given by \eqref{Rr}, boundary datum $g=g_3$}
\end{figure}

\begin{figure}[!ht]
\centering
\includegraphics[height=3cm, width=16cm]{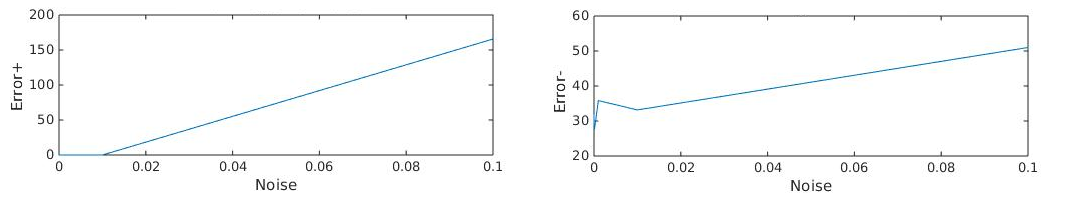}

\caption{Error+ (surplus of emitters), Error- (missed emitters) against noise. Results obtained by \textbf{Algorithm 1}, $\tau=.5, \lambda=10^{-6}$. }
\end{figure}
\begin{figure}
\centering
\includegraphics[height=3cm, width=16cm]{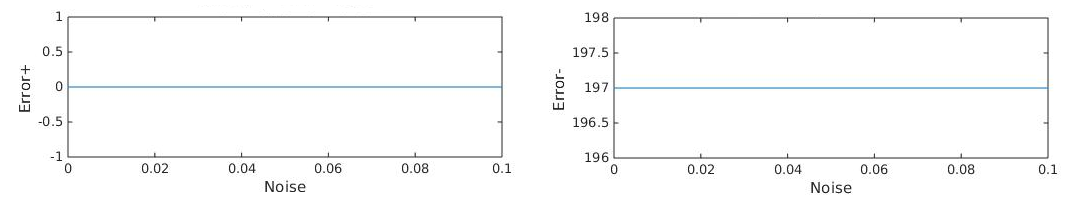}

\caption{Error+ (surplus of emitters), Error- (missed emitters) against noise. Results obtained by GIST, $\tau=.5, \lambda=10^{-6}$.}
\end{figure}

\begin{figure}[!ht]
\centering

\includegraphics[height=3cm, width=16cm]{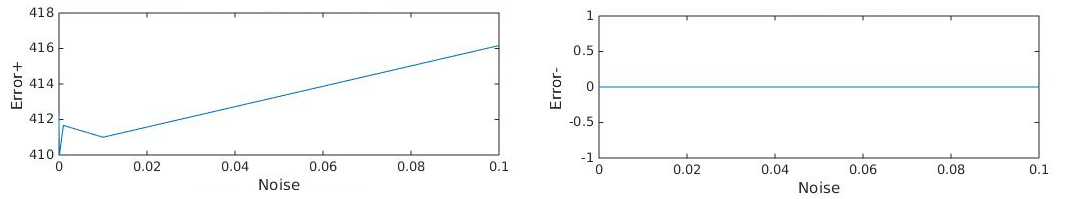}

\caption{Error+ (surplus of emitters), Error- (missed emitters) against noise. Results obtained by the FISTA, $ \lambda=10^{-6}$.}
\end{figure}





\begin{thebibliography}{}
\bibitem{A1} M. Artina, M. Fornasier, F. Solombrino, \textit{Linearly constrained nonsmooth and nonconvex minimization}, SIAM J. Optim., 23,  1904-1937 (2013).
\bibitem{A2} M. Artina, F. Cagnetti, M. Fornasier, F. Solombrino, \textit{Linearly constrained evolution of critical points and an application to cohesive fractures}, Math. Models Methods Appl. Sci., 27(02):231-290,2017.
\bibitem{BMYX} H. P. Babcock, J. R. Moffitt, Y. Cao, X. Zhuang, \textit{Fast compressed sensing analysis for super-resolution
imaging using L1-homotopy}, Optics Express, 21  (2013), pp. 28583-28596.
\bibitem{BA} G. I. Barenblatt, \textit{The mathematical theory of equilibrium cracks in brittle fracture}, Adv. Appl. Math. Mech. 7 (1962), pp. 55-129.
\bibitem{BPS} E. Betzig, G. H. Patterson, R. Sougrat, O. W. Lindwasser, S. Olenych, J. S. Bonifacino, M. W. Davidson, J.
Lippincott-Schwartz, H. F. Hess, \textit{Imaging intracellular fluorescent proteins at nanometer resolution},
Science, 313 (2006), pp. 1642-1645.
\bibitem{BRLORE} K. Bredies, D.A. Lorentz, S. Reiterer, \textit{Minimization of non-smooth, nonconvex functionals by iterative thresholding}, J. Optim. Theory Appl., 165 (2015), pp. 78-112.
\bibitem{BRHU} P. Breheny, J. Huang. \textit{Coordinate descent algorithms for nonconvex penalized regression, with applications to biological feature selection}, The annals of applied statistics, Vol.5, No. 1, 232-253 (2011).
\bibitem{7} E. Candes, T. Tao, \textit{Decoding by linear programming}, IEEE Trans. Inform. Theory, 51(12):4203-4215,2005.
\bibitem{C} E. Candes, J. Romberg, T. Tao, \textit{Stable Signal Recovery from Incomplete and Inaccurate Measurements}, Comm. Pure Appl. Math., 59 (2006), pp. 1207-1223.
\bibitem{CS} R. Chartrand, V. Staneva, Restricted isometry properties and nonconvex compressive sensing, Inverse Problems, 24(3):035020, 14 (2008).
\bibitem{10} S.S. Chen, D. L. Donoho, M. A. Saunders, \textit{Atomic decomposition by basis pursuit}, SIAM J. Sci. Comput., 20(1):33-61,1998. 
\bibitem{14} D. L. Donoho, \textit{Compressed sensing}, IEEE TRans. Inform. Theory, 52(4):1289-1306,2006.
\bibitem{DG} D. S. Dugdale, \textit{Yielding of steel sheets containing slits}, J. Mech. Phys. Solids,   8 (1960), pp. 100-104.
\bibitem{DP} V. Duval, G. Peyr\'e, \textit{Exact support recovery for sparse spikes deconvolution}, Found. Comput. Math., 15 (2015), pp. 1315-1355.
\bibitem{FANFENG} J. Fan, Y. Feng, Y. Wu, \textit{ Network exploration via the adaptive LASSO and SCAD penalties}, Ann. Appl. Stat., 3(2):521-541 (2009).
\bibitem{FL} J. Fan, R. Li, Variable selection via nonconcave penalized likelihood and its oracle properties, J. Amer. Statist. Assoc., 96(456):1348-1360 (2001).
\bibitem{FP} J. Fan, H. Peng, Nonconcave penalized likelihood with a diverging number of parameters, Ann. Statist., 32(3):928-961 (2004).
\bibitem{FLai} S. Foucart, M.-J- Lai, Sparsest solutions of underdetermined linear szstems via $\ell_q$-minimization for $0<q\leq 1$, Appl. Compt. Harmon. Anal., 26(3):395-407 (2009).
\bibitem{GKproc} D. Ghilli, K. Kunisch, \textit{A  monotone scheme for sparsity optimization in $\ell^p$ with $p\in (0,1]$}, to appear in IFAC  WC 2017 Proceedings.
\bibitem{GK} D. Ghilli, K. Kunisch, \textit{On monotone and primal dual active set schemes for sparsity optimization in $\ell^p$ with $p\in (0,1]$}, preprint $2017$, arxiv:1709.06506.
\bibitem{GIST} P. Gong, C. Zang, Z. Lu, J. Z. Huang, J. Ye, \textit{A general iterative shrinkage and thresholding algorithm for non-convex regularized optimization problems}, Proceeding ICML'13 Proceedings of the 30th International Conference on International Conference on Machine Learning-Volume 28, Pages II 37-II-45., Atlanta, GA, USA-June 16-21, 2013.
\bibitem{GSC} L. Gu, Y. Sheng, Y. Chen, H. Chang,Y. Zhang, P. Lv, W. Ji, T. Xu, \textit{High-Density 3D single molecular  analysis based on compressed sensing}, Biophysical Journal, 106 (2014), pp. 2443-2449.
\bibitem{HGM} S. T. Hess, T. P. Girirajan, M. D. Mason, \textit{Ultra-high resolution imaging by fluorescence photoactivation
localization microscopy}, Biophysical Journal, 91 (2006), pp. 4258-4272.
\bibitem{HIWU13} M. Hinterm{\"u}ller, Tao Wu, \textit{Nonconvex $TV^q$-models in image restoration: analysis and a trust-region regularization-based superlinearly convergent solver}, SIAM J. Imaging Sci., 6 (2013), pp. 1385-1415.
\bibitem{HHM} J. Huang, J. L. Horowitz, S. Ma, Asymptotic properties of bridge estimators in sparse high-dimensional regression models, Ann. Statist., 26(2):587-613 (2008).
\bibitem{HBZ} B. Huang, H. P. Babcock,  X. Zhuang, \textit{Breaking the diffraction barrier: super-resolution imaging of cells},
Cell, 143 (2010), pp. 1047-1058.
\bibitem{KI} K. Ito, K. Kunisch, \textit{A Variational approach to sparsity optimization based on
Lagrange multiplier theory}, Inverse Problems, 30 (2014),  015001, 23pp.
\bibitem{J} Y. Jiao, B. Jin, X. Lu, \textit{A primal dual active set with continuation algorithm for the $\ell^0$-regularized optimization problem}, Appl. Comput. Harmon. Anal., 39 (2015), pp. 927-957.
\bibitem{JJLR} Y. Jiao, B. Jin, X. Lu, W. Ren, \textit{A primal dual active set algorithm for a class of nonconvex sparsity optimization}, Preprint, (2013).
\bibitem{KKR} D. Kalise, K. Kunisch, Z. Rao, \textit{Infinite horizon sparse optimal control},  J. Optim. Theory Appl., 172 (2017), pp. 481-517.
\bibitem{KKF} K. Knight, W. Fu, Asymptotics for lasso-type estimators, Ann. Statist., 28(5):1356-1378 (2000).
\bibitem{LI} G. Li, T.K. Pong, \textit{Global convergence of splitting methods for nonconvex composite optimization}, SIAM J. Optim., 25 (2014), pp. 2434-2460.
\bibitem{LLSZ} Q. Lyu, Z. Lin, Y. She, C. Zhang, \textit{A comparison of typical $\ell^p$ minimization algorithms}, Neurocomputing 119 (2013), 413-424.
\bibitem{PI13} G. Del Piero, \textit{ A variational approach to fracture and other inelastic phenomena},  J. Elasticity, 112 (2013), pp. 3-77.
\bibitem{OHDOBRPO15} P. Ochs, A. Dosovitskiy, T. Brox, T. Pock, \textit{On iteratively reweighted algorithms for nonsmooth nonconvex optimization in computer vision}, SIAM J. Imaging Sci.,  8 (2015), pp. 331-372.
\bibitem{RBZ} M. Rust, M. Bates, X. Zhuang, \textit{Sub-diffraction-limit imaging by stochastic optical reconstruction
microscopy (STORM)}, Nature Methods, 3 (2006), pp. 793-796.
\bibitem{45} Q. Sun, \textit{Recovery of sparsest signals via $\ell^q$-minimization}, Appl. COmput. Harmon. Anal., 32(3):329-341,2012. 
\bibitem{tuia} D. Tuia, R\'emi Flamary, Michel Barlaud, \textit{Non-convex regularization in remote sensing}, IEEE Transactions on Geoscience and Remote Sensing, Institute of Electrical and Electronics Engineers, 2016.
\bibitem{46} R. Tibshirani, \textit{Regression shrinkage and selection via the lasso}, J. Roy. Statist. Soc. Ser. B, 58(1):267-288,1996.
\bibitem{CHZ} C.-H. Zhang, Nearly unbiased variable  selection under minimax concave penalty, Ann. Statist., 38(2):894-942 (2010).
\bibitem{51} C.-H. Zhang, J. Huang, \textit{The sparsity and bias of the LASSO selection in high-dimensional sparse estimation problems}, Statist. Sci, 27(4):576-593,2012.
\bibitem{ZZEH} L. Zhu, W. Zhang, D. Elnatan,  B. Huang, \textit{Faster STORM using compressed sensing}, Nature Methods, 9 (2012), pp. 721-723.


\end{thebibliography}
\end{document}